\documentclass[usenames,dvipsnames,reqno,10pt]{amsart}

\usepackage[a4paper,left=30mm,right=30mm,top=30mm,bottom=30mm,marginpar=25mm]{geometry}
\usepackage{amsmath}
\usepackage{amssymb}
\usepackage{amsthm}
\usepackage{amscd}
\usepackage[ansinew]{inputenc}
\usepackage[final]{graphicx}
\usepackage{bbm}
\usepackage{subcaption}
\usepackage{tikz}
\usetikzlibrary{calc}
\usepackage{xcolor}
\usepackage{cases}

\usepackage[colorlinks=true, linkcolor=blue, citecolor=blue]{hyperref}
\hypersetup{final}

\renewcommand{\epsilon}{\varepsilon}

\numberwithin{equation}{section}

\newtheoremstyle{thmlemcorr}{10pt}{10pt}{\itshape}{}{\bfseries}{.}{10pt}{{\thmname{#1}\thmnumber{ #2}\thmnote{ (#3)}}}
\newtheoremstyle{thmlemcorr*}{10pt}{10pt}{\itshape}{}{\bfseries}{.}\newline{{\thmname{#1}\thmnumber{ #2}\thmnote{ (#3)}}}
\newtheoremstyle{defi}{10pt}{10pt}{\itshape}{}{\bfseries}{.}{10pt}{{\thmname{#1}\thmnumber{ #2}\thmnote{ (#3)}}}
\newtheoremstyle{remexample}{10pt}{10pt}{}{}{\bfseries}{.}{10pt}{{\thmname{#1}\thmnumber{ #2}\thmnote{ (#3)}}}
\newtheoremstyle{ass}{10pt}{10pt}{}{}{\bfseries}{.}{10pt}{{\thmname{#1}\thmnumber{ A#2}\thmnote{ (#3)}}}

\theoremstyle{thmlemcorr}
\newtheorem{theorem}{Theorem}
\numberwithin{theorem}{section}
\newtheorem{lemma}[theorem]{Lemma}
\newtheorem{corollary}[theorem]{Corollary}
\newtheorem{proposition}[theorem]{Proposition}

\theoremstyle{thmlemcorr*}
\newtheorem{theorem*}{Theorem}
\newtheorem{lemma*}[theorem]{Lemma}
\newtheorem{corollary*}[theorem]{Corollary}
\newtheorem{proposition*}[theorem]{Proposition}
\newtheorem{problem*}[theorem]{Problem}
\newtheorem{conjecture*}[theorem]{Conjecture}

\theoremstyle{defi}
\newtheorem{definition}[theorem]{Definition}

\theoremstyle{remexample}
\newtheorem{remark}[theorem]{Remark}
\newtheorem{example}[theorem]{Example}

\theoremstyle{ass}

\newcommand{\Bcal}{\mathcal{B}}
\newcommand{\Ccal}{\mathcal{C}}
\newcommand{\Dcal}{\mathcal{D}}

\newcommand{\Fcal}{\mathcal{F}}
\newcommand{\Gcal}{\mathcal{G}}
\newcommand{\Hcal}{\mathcal{H}}

\newcommand{\Mcal}{\mathcal{M}}
\newcommand{\Ncal}{\mathcal{N}}

\newcommand{\Scal}{\mathcal{S}}
\newcommand{\Tcal}{\mathcal{T}}

\DeclareMathOperator{\dist}{dist}

\DeclareMathOperator{\supp}{supp}

\newcommand{\abs}[1]{|#1|}

\newcommand{\dd}{\;\mathrm{d}}

\newcommand{\N}{\mathbb{N}}
\newcommand{\R}{\mathbb{R}}

\newcommand{\ONE}{\mathbbm{1}}

\newcommand{\weaklystar}{\overset{*}\rightharpoonup}

\newcommand{\eps}{\epsilon}

\newcommand{\bdpoly}{\partial(\Omega,R_\ast)}

\newcommand{\bdperp}{\partial_{\perp}(\Omega, R_\ast)}

\def\XXint#1#2#3{{\setbox0=\hbox{$#1{#2#3}{\int}$}
\vcenter{\hbox{$#2#3$}}\kern-.5\wd0}}

\newcommand\restrict[1]{\raisebox{-.5ex}{$|$}_{#1}}
\DeclareMathOperator{\SO}{SO}
\DeclareMathOperator{\Sl}{Sl}
\DeclareMathOperator{\Id}{Id}

\DeclareMathOperator{\Diss}{Diss}

\newcommand{\ffi}{\varphi}
\newcommand{\rc}{{\rm rc}}
\newcommand{\qc}{{\rm qc}}
\newcommand{\pc}{{\rm pc}}

\title[Analysis of polycrystalline single-slip crystal plasticity]{On the interplay of anisotropy and geometry for polycrystals in single-slip crystal plasticity}

\author{Dominik Engl}
\address{Mathematical Institute, Utrecht University, Postbus 80010, 3508 TA Utrecht, The Netherlands}
\email{d.m.engl@uu.nl}
\author{Carolin Kreisbeck}
\address{Mathematisch-Geographische Fakult\"at, Katholische Universit\"at Eichst\"att-Ingolstadt, Osten\-stra{\ss}e 28, 85072 Eichst\"att}
\email{carolin.kreisbeck@ku.de}

\begin{document}

\maketitle

\begin{abstract}  
\vspace{-12pt}
In this paper, we investigate a variational polycrystalline model in finite crystal plasticity with one active slip system and rigid elasticity.
The task is to determine inner and outer bounds on the domain of the constrained macroscopic elastoplastic energy density, or equivalently, the affine boundary values of a related inhomogeneous differential inclusion problem. 
A geometry-independent Taylor inner bound, which we calculate directly, follows from considering constant-strain solutions to a relaxed problem in combination with well-known relaxation and convex integration results.
On the other hand, we deduce outer bounds from a rank-one compatibility condition between the affine boundary data and the microscopic strain at the boundary grains.
While there are examples of polycrystals for which the two above-mentioned bounds coincide, we present an explicit construction to prove that the Taylor bound is non-optimal in general. 
\bigskip

\thispagestyle{empty}

\smallskip

 \smallskip

\noindent\textsc{Date:} \today.
\end{abstract}

\section{Introduction}

Most elastoplastic solids are polycrystalline, meaning that they consist of rotated copies of single crystals, called grains, which form patterns on a mesoscopic length scale in between the micro- and the macroscopic one. The grain structures impose restrictions on still finer substructures and highly influence the effective material response of the solid. 
In this paper, we contribute to the analysis of the attainable macroscopic strains of polycrystals, focusing on a model of crystal plasticity with one active slip-system and rigid elasticity.  
The overall goal is to obtain a deeper understanding of boundary interaction, global compatibility, and the interplay between slip mechanisms and texture in a geometrically nonlinear setting.

Our approach originates in the time-discrete variational framework introduced in~\cite{CHM02, MSL02, OrR99} for modeling rate-independent processes arising in single-crystal finite plasticity. 
The first relaxation results for such problems, facilitating the description of the effective material behavior by minimizing out microstructural effects, go back to Conti \& Theil~\cite{Con06, CoT05}. They determine the quasiconvexification of the elastoplastic energy density for the first time-incremental problem under the assumptions of one single slip system and elastically rigid behavior.
As shown in~\cite{CDK11}, the relaxed energies of~\cite{Con06} result via approximation by $\Gamma$-convergence from models with elastic energy in the limit of diverging elastic constants.
For more recent work on relaxation in models with two or more slip systems, we refer to~\cite{CoD16, CDK13a, Sch13}; see also~\cite{ADK18, AnD15} for related studies in the context of strain-gradient plasticity.

Compared with the single-crystal case, the analysis of polycrystals holds additional challenges related to the geometry and orientation of the different grains, as well as to the compatibility of micro\-structures across grain boundaries; in the context of linear and nonlinear elasticity, the latter has been studied in~\cite{Bha03,BaK97} and~\cite{BaC15, BaC17}, respectively.  
The description of polycrystal geometry in~\cite{BaC15} by Ball \& Carstensen in combination with the modeling of~\cite{Con06, CoT05} constitutes the basis for our framework of
polycrystalline finite crystal plasticity with one active slip system under the assumption of elastically nonlinear but rigid behavior; the detailed setup is given in Section~\ref{sec:setup}. 
Let us remark that the stress is not well-defined in this elastically rigid strain-based setting; for an analysis of a stress-based formulation of polycrystal perfect plasticity, see~\cite{KoL99}.
Mathematically, the macroscopically attainable strains can be described via a specific inhomogeneous nonlinear differential inclusion subject to affine boundary conditions, similar to the approaches in polycrystalline shape-memory materials \cite{BaK97,KoN00}, where stress-free strains are studied based on elastic energy minimization.
The general theory of differential inclusions  (or multi-valued differential equations) has been an active field of research over the last decades with strong methods and results, which are especially rich in the homogeneous case, see, e.g.,~\cite{Mue99, Syc06, Syc11} and the references therein. In the inhomogeneous case, where the target sets feature spatial dependence, we refer to~\cite{MuS01} for the existence of Lipschitz solutions, and to~\cite{MaS21} for a recent generalization of the latter in the context of Sobolev solutions. 
The techniques of~\cite{KRW15} allow, among others, the construction of Sobolev functions with prescribed Jacobians under a so-called uniform tight containedness assumption. 
However, there is - to the best of our knowledge - currently no available abstract methodology that is able to accommodate the particular setting of this paper - even though, considering related homogeneous inclusions does provide some partial insight.

This work addresses different aspects related to the solvability of inhomogeneous differential inclusions used to describe the macroscopic deformation behavior of elastoplastic solids.
We first identify a simple geometry-independent sufficient condition by combining a new characterization of globally affine solutions to a relaxed version of the problem with well-known relaxation and convex integration results~\cite{CoT05,MuS99}. 
On the other hand, necessary conditions are due to compatibility constraints following from a generalized Hadamard jump condition~\cite{BaC, IVV02} applied to the boundary grains. 
While the sufficient and necessary conditions turn out to provide a characterization for specific polycrystals, we show that they do not coincide in general though.  
The argument is based on an explicit construction of finitely piecewise affine maps that satisfy fixed boundary conditions and incompressibility, and is as such known to be a delicate issue. Here, we take the geometric setup of the rotated-square construction in \cite{CKZ17, Pom10} as inspiration for a suitable sheared-square construction. A more detailed overview of our findings is given in Section~\ref{sec:overview}.

\subsection{Setup of the problem}\label{sec:setup}
In the following, we describe our two-dimensional model for single-slip polycrystal plasticity. 
Starting from the theory of finite plasticity for single-crystalline structures (cf.~\cite{CHM02, MSL02, Mie07, OrR99}), we adopt a geometrically nonlinear model for the deformation behavior of an elastoplastic body where the deformation gradient 
\begin{align}\label{decomposition}
	\nabla u=F=F_{\rm el}F_{\rm pl}
\end{align} is split multiplicatively into an elastic part $F_{\rm el}$ and a plastic one $F_{\rm pl}$, as proposed in~\cite{Kro60, Lee69}. For recent discussions about this decomposition and possible alternative modeling approaches, see~\cite{CDOR18, DaF15, Del18}.
With~\eqref{decomposition} at hand, the elastoplastic energy is given in terms of the condensed energy density  
\begin{align}\label{densityW}
	W(F) = \min_{F=F_{\rm el}F_{\rm pl}}\bigr( W_{\rm el}(F_{\rm el}) + W_{\rm pl}(F_{\rm pl}) + \Diss(F_{\rm pl})\bigr),
\end{align}
where $W_{\rm el}$ is the elastic energy contribution, $W_{\rm pl}$ represents the plastic potential, and $\Diss$ encodes dissipative effects.
We invoke a setting of rigid elasticity,
meaning that the elastic parts $F_{\rm el}$ are contained in the set of rotations $\SO(2)$ almost everywhere and do not contribute to the energy, i.e.,
\begin{align*}
	W_{\rm el}(F_{\rm el})= \begin{cases}
				0 &\text{ if } F_{\rm el}\in\SO(2),\\
				\infty & \text{ otherwise.}
			\end{cases}
\end{align*}
As the plastic strain in the context of single-slip crystal plasticity is a simple shear along one active slip system, determined by a slip direction $s\in\Scal^1$ and slip-plane normal $m=s^\perp$, one has that $F_{\rm pl} = \Id + \gamma s \otimes m$,
where $\gamma\in\R$ quantifies the plastic slip, and~\eqref{densityW} becomes 
\begin{align}\label{Whomogenous}
	W(F) = 	\begin{cases}
				(|Fm|^2 - 1)^{\frac{p}{2}} = |\gamma|^p &\text{ if } F\in\Mcal_s,\\
				\infty &\text{ otherwise,}
			\end{cases}\qquad F\in\R^{2\times 2},
\end{align}
for $p\geq 1$, cf.~\cite{Con06, CoT05}. Here, the choices $p=1$ and $p=2$ model dissipation and linear hardening, respectively, and 
\begin{align*}
	\Mcal_s:=\{F\in \R^{2\times 2} : \det F=1, |Fs|= 1\};
\end{align*}
we often use the short notation $\Mcal:=\Mcal_{e_1}$.

In our polycrystalline setting, the active slip direction within the body is location-dependent and determined by the orientations of the individual grains. 
To be more precise, let us first give a definition of the term ``polycrystal'', which is based on~\cite[Section 2]{BaC15} and tailored to our crystal plasticity model, see Figure \ref{fig:polycrystal} for illustration.
We say that a pair $(\Omega, R_\ast)$ with a reference configuration $\Omega\subset \R^2$ and a texture $R_\ast: \Omega\to \SO(2)$ is a (two-dimensional) polycrystal if these conditions are satisfied: 
\begin{itemize}
	\item $\Omega$ is a bounded Lipschitz and has a partition (up to a set of measure zero) into $N\in\N$ regular bounded Lipschitz domains $\Omega_1,\ldots,\Omega_N\subset \Omega$, called the grains of $(\Omega,R_\ast)$, that is,
	\begin{align*}
		\Omega={\rm int}\:\bigcup_{k=1}^N\overline{\Omega}_k\quad\text{ and }\quad{\rm int}\:\overline{\Omega}_k = \Omega_k \text{ for all $k\in\{1,\ldots,N\}$},
	\end{align*}
	where $\overline{(\cdot)}$ and ${\rm int}\: (\cdot)$ denote the closure and interior of a set; 
	
	\item $R_\ast : \Omega \to \SO(2)$ is constant on each $\Omega_k$ for $k\in\{1,\ldots,N\}$ with
	\begin{align}\label{twograins}
		R_\ast\restrict{\Omega_k}\neq \pm R_\ast\restrict{\Omega_l} \text{ if } \Hcal^1(\partial\Omega_k\cap\partial\Omega_l)>0 \text { for all }k\neq l \in \{1,\ldots,N\},
	\end{align}
	where $\Hcal^1$ is the one-dimensional Hausdorff-measure; otherwise, two grains can be merged into one.
	The image of $R_\ast$ is denoted by $R_\ast(\Omega)$. 
\end{itemize}
With the default choice of the slip direction $e_1$ on unrotated grains of the polycrystal $(\Omega,R_\ast)$, each grain $\Omega_k$ can be viewed as a single crystal with slip direction $R_\ast\restrict{\Omega_k}e_1$. 
The set of attainable microscopic strains on $\Omega_k$ is then given by
\begin{align*}
	\Mcal_{R_\ast\restrict{\Omega_k}e_1} = \Mcal R_\ast^T\restrict{\Omega_k}. 
\end{align*}
Note that polycrystals with exactly two grains differ intrinsically from single-crystals with two active slip systems subject to latent hardening, as studied in~\cite{CDK13a,Sch13}. This is because the slip directions are tied to the grain structure and may not be chosen arbitrarily at any point within the occupied region.

We distinguish between interior and boundary grains: the boundary of an interior grain $\Omega_k$ is contained in the boundary of others, i.e., $\partial \Omega_k\subset \bigcup_{l\neq k} \partial\Omega_l$; all remaining grains are called boundary grains. 
Furthermore, we introduce the points on $\partial \Omega$ where at least two boundary grains meet as the boundary dual points of $(\Omega, R_\ast)$, in formulas,
\begin{align}\label{bddual_def}
	\bigcup_{1\leq k<l\leq N} \partial \Omega_{k}\cap \partial\Omega_{l}\cap \partial \Omega.
\end{align}
With $\nu$ the outer unit normal of $\Omega$, we set
\begin{align}\label{bdpoly_def}
	\bdpoly:=\{x\in\partial\Omega : \text{$\nu(x)$ exists and $x$ is not a boundary dual point of $(\Omega, R_\ast)$}\}.
\end{align}

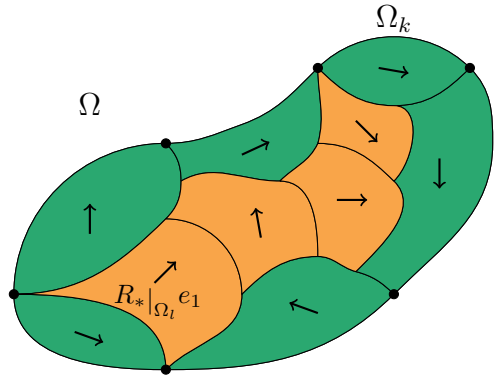
\begin{figure}[h!]
	\centering
	\begin{tikzpicture}
		\begin{scope}[scale=1]
			\draw (0,0) to [out=90,in=180] (2,2) to [out=0,in=200] (3,2.25) to [out=20,in=225] (4,3) to [out=45, in=135] (6,3) to [out=-45, in=90] (6.3,2) to [out=-90,in=45] (5,0) to [out=225,in=0] (2,-1) to [out=180,in=-90] (0,0); 
			\draw (1,2.25) node[anchor = south] {\Large{$\Omega$}};	
			
			\draw [fill=ForestGreen!80!white](0,0) to [out=0,in=225] (2,1) to [out=45,in=-90] (2.2,1.5) to [out=90,in=-45] (2,2) to [out=180,in=90] (0,0); 
			\draw [ thick, ->] (1,0.8) -- ++(0,0.4);
			\draw [fill=ForestGreen!80!white](0,0) to [out=0,in=90](2,-1) to [out=180,in=-90] (0,0); 
			\draw [ thick, ->](0.8,-0.5) --++(-20:0.4);
			\draw [fill=ForestGreen!80!white](2.2,1.5) to [out=90,in=-45] (2,2) to [out=0,in=200] (3,2.25) to [out=20,in=225] (4,3) to [out=-90,in=75] (4,2) to [out=-115,in=45] (3.5,1.5) to [out=180,in=30] (2.2,1.5);
			\draw [ thick, ->] (3,1.85) -- ++(25:0.4);
			\draw [fill=ForestGreen!80!white](3,0) to [out=45,in=180] (4,0.5 )to [out=0,in=165] (4.5,0.3) to [out=-15,in=135] (5,0) to [out=225,in=0] (2,-1) to [out=45,in=-90] (3,0); 
			\draw [ thick, ->] (4,-0.25) -- ++(160:0.4);
			\draw [fill=ForestGreen!80!white](6,3) to [out=-45, in=90] (6.3,2) to [out=-90,in=45] (5,0) to [out=135,in=-15] (4.5,0.3) to [out=45,in=-45] (5,1.5) to [out=45,in=0] (5,2.5) to [out=0,in=225] (6,3);
			\draw [ thick, ->] (5.6,1.8) --++(-90:0.4);
			\draw [fill=ForestGreen!80!white](4,3) to [out=45, in=135] (6,3) to [out=225,in=0] (5,2.5) to [out=180,in=-45] (4,3); 
			\draw [ thick, ->] (4.8,3) --++(-10:0.4);
			\draw (5,3.35) node[anchor =south] {\Large{$\Omega_k$}};

			\draw [fill=BurntOrange!80!white,line join=round](0,0) to [out=0,in=225] (2,1) to [out=0,in=90] (3,0) to [out=-90,in=45] (2,-1) to [out=90,in=0] (0,0); 
			\draw (1.90,-0.1) node {{$R_\ast\restrict{\Omega_l}e_1$}};
			\draw [ thick, ->] (1.85,0.15) -- ++(45:0.4cm);
			\draw [fill=BurntOrange!80!white](2,1) to [out=45,in=-90] (2.2,1.5) to [out=30,in=180] (3.5,1.5) to [out=0,in=90] (4,0.5) to [out=180,in=45] (3,0) to [out=90,in=0] (2,1); 
			\draw [ thick, ->] (3.25,0.75) --++(100:0.4);
			\draw [fill=BurntOrange!80!white](3.5,1.5) to [out=45,in=-115] (4,2) to [out=0,in=135] (5,1.5) to [out=-45,in=45] (4.5,0.3) to [out=165,in=0] (4,0.5) to [out=90,in=0] (3.5,1.5); 
			\draw [ thick, ->] (4.25,1.25) --++(0:0.4);
			\draw [fill=BurntOrange!80!white,line join=round](4,2) to [out=75,in=-90] (4,3) to [out=-45,in=180] (5,2.5) to [out=0,in=45] (5,1.5) to [out=135,in=0] (4,2);
			\draw [ thick, ->] (4.5,2.3) --++(-45:0.4);
			
			\draw [fill=black](0,0) circle (0.6mm);
			\draw [fill=black](2,2) circle (0.6mm);
			\draw [fill=black](4,3) circle (0.6mm);
			\draw [fill=black](6,3) circle (0.6mm);
			\draw [fill=black](5,0) circle (0.6mm);
			\draw [fill=black](2,-1) circle (0.6mm);			
		\end{scope}
	\end{tikzpicture}
	\caption{A visualization of a polycrystal $(\Omega,R_\ast)$, where the green and orange domains are boundary and interior grains, respectively; the arrows indicate the orientation of the slip directions $R_\ast e_1$ and the black dots at the boundary highlight the boundary dual points.}\label{fig:polycrystal}
\end{figure}

A variational approach to the deformation behavior of a polycrystal $(\Omega, R_\ast)$ requires the inhomogeneous microscopic energy density
\begin{align}\label{Wpoly_micro}
	W(x,F) = 	\begin{cases}
							(|FR_\ast(x)e_2|^2-1)^{\frac{p}{2}} &\text{ if } F\in \Mcal R_\ast^T(x),\\
							\infty &\text{ otherwise,}
						\end{cases}\quad x\in\Omega, \ F\in\R^{2\times 2},
\end{align}
which results from rotated versions of the corresponding density in the single-crystalline case~\eqref{Whomogenous}.  
Following the work by Bhattacharya \& Kohn~\cite{BaK97}, we define the macroscopic energy density via averages of the microscopic one and optimization over all possible microstructures forming within the grains, that is, 
\begin{align}\label{Wpoly_macro}
	W_{(\Omega,R_\ast)}(F) = \inf_{\substack{u\in W^{1,\infty}(\Omega;\R^2) \\  u=Fx \text{ on }\partial\Omega}} \frac{1}{|\Omega|}\int_{\Omega} W(x, \nabla u) \dd x,\quad F\in\R^{2\times 2}.
\end{align}

In this article, our goal is to characterize (inner and outer bounds on) the domain of $W_{(\Omega,R_\ast)}$, which involves a deeper understanding of the inhomogeneous partial differential inclusion
\begin{align*}
	\begin{cases}
		\nabla u(x) \in \Mcal R_\ast^T(x) &\text{ for a.e.~} x\in \Omega,\\
		u(x) = Fx &\text{ for } x\in\partial\Omega,
	\end{cases}\tag{$P_\Mcal$}\label{unrelaxed}
\end{align*}
where $u\in W^{1,\infty}(\Omega;\R^2)$ is the unknown and $F\in \R^{2\times 2}$; in fact, $W_{(\Omega,R_\ast)}$ is finite exactly on
\begin{align*}
	\Fcal_{\Mcal}(\Omega,R_\ast)&:=\{F\in \R^{2\times 2}:\text{ there exists a solution }u\in W^{1,\infty}(\Omega; \R^2)\text{ to \eqref{unrelaxed}}\},
\end{align*}
considering \eqref{Wpoly_micro} and \eqref{Wpoly_macro}. We point out that the special linear group of degree two constitutes a trivial outer bound 
\begin{align}\label{Sl2_bound}
	\Fcal_\Mcal(\Omega,R_\ast)\subset \Sl(2):= \{F\in \R^{2\times 2} : \det F = 1\}
\end{align}
for any polycrystal $(\Omega,R_\ast)$. This follows from the observations that the solutions to \eqref{unrelaxed} are locally volume-preserving and that the determinant is a null Lagrangian.

In the case that $(\Omega,R_\ast)$ is a single-crystal, i.e.,  the texture $R_\ast$ is constant on all of $\Omega$, the set of attainable macroscopic strains has been identified via a relaxation-type argument  
in~\cite{CoT05} as $\Fcal_\Mcal(\Omega,R_\ast) = \Ncal_s$ with $s=R_\ast e_1$ and 
\begin{align}\label{Ns}
	\Ncal_s:=\{F \in \R^{2\times 2} : \det F=1, |Fs|\leq 1\} = \Mcal_s^\pc = \Mcal_s^\qc = \Mcal_s^\rc,
\end{align}
where $\Mcal^\qc$ ($\Mcal^\pc$, $\Mcal^\rc$) is the quasiconvex (polyconvex, rank-one convex) hull of $\Mcal_s$, see ~\eqref{qc_hull} for the definitions; we set $\Ncal:=\Ncal_{e_1}$.
The proof technique of~\cite{CoT05} exploits - besides classical results on homogeneous partial differential inclusions with Lipschitz solutions (e.g.~\cite[Theorem 4.10]{Mue99}) - the seminal work by M{\"u}ller \& S{\v v}erak~\cite{MuS99}, which in turn extends Gromov's theory~\cite{Gro73,Gro86} of convex integration and its applications. 
This relaxation result in the single-crystal case motivates to study also the relaxed inclusion problem
\begin{align*}
	\begin{cases}
		\nabla u(x) \in \Ncal R_\ast^T(x) &\text{ for a.e.~$x\in\Omega$},\\
		u(x) = Fx &\text{ for $x\in\partial\Omega$},
	\end{cases}\tag{$P_\Ncal$}\label{relaxed}
\end{align*}
with unknown $u\in W^{1,\infty}(\Omega;\R^2)$ and $F\in \R^{2\times 2}$, to gain insight into the structure of $\Fcal_{\Mcal}(\Omega,R_\ast)$. 

The fact that~\eqref{unrelaxed} and~\eqref{relaxed} are non-convex inhomogeneous differential inclusion problems whose target set is unbounded makes them non-standard. In particular, the results of~\cite{KRW15, MaS21, MuS01} are not applicable and, naturally, the classical theory of homogeneous inclusions is limited in providing useful new insight. Given that the solvability of \eqref{unrelaxed} and \eqref{relaxed} depends fundamentally
on the interaction between the shape, size and orientation of the grains makes the analysis many-faceted.

\subsection{Overview of the main results}\label{sec:overview}
Throughout this paper, let $(\Omega,R_\ast)$ be a polycrystal.
To analyze the set of attainable macroscopic strains $\Fcal_\Mcal(\Omega,R_\ast)$, we identify inner and outer bounds, and show that they coincide under suitable conditions on the texture. \smallskip

Two geometry-independent inner bounds are tied to the assumption of constant strain, which can be traced back to the early works by Taylor~\cite{Tay38} and Bishop \& Hill~\cite{BiH51}.
Precisely, we consider the sets of globally affine solutions to \eqref{unrelaxed} and \eqref{relaxed}, given by the finite intersections 
\begin{align}\label{taylor_sets}
	\Tcal_\Mcal(R_\ast(\Omega)) := \bigcap_{x\in\Omega} \Mcal R_\ast^T(x) \qquad
\text{and} \qquad \Tcal_\Ncal(R_\ast(\Omega)) := \bigcap_{x\in\Omega} \Ncal R_\ast^T(x);
\end{align}
notice that these sets are independent of the size and shape of the grains and only take the orientation of the slip systems into account. 
Based on the work on single-crystal plasticity by Conti \& Theil \cite{CoT05} (see also Proposition \ref{prop:affine_lipschitz}), it holds that
\begin{align}\label{taylor_inclusions}
	\Tcal_\Mcal(R_\ast(\Omega)) \subset \Tcal_\Ncal(R_\ast(\Omega)) \subset \Fcal_\Mcal(\Omega,R_\ast).
\end{align}
We refer to $\Tcal_\Mcal(R_\ast(\Omega))$ and $\Tcal_\Ncal(R_\ast(\Omega))$ as the Taylor bound for the differential inclusions \eqref{unrelaxed} and \eqref{relaxed}, in analogy to the terminology in~\cite[Section 2.4]{BaK97} and~\cite[Proposition 2.1]{KoN00} on polycrystalline shape-memory materials.

Our first main result shows that $\Tcal_\Ncal(R_\ast(\Omega))$ depends, in fact, on at most three specific orientations of the polycrystal.
\begin{proposition}[Characterization of the Taylor bound for~\eqref{relaxed}]\label{prop:taylorbound}
	Suppose that $R_\ast$ attains the values
	\begin{align}\label{ordering}
		R_{\theta_1},\ldots,R_{\theta_N} \quad \text{for $0=\theta_1< ... < \theta_N < \pi$ with $N\geq 2$}
	\end{align} 
	on the grains of $(\Omega,R_\ast)$, and let $\theta_{N+1} = \pi$. Then,   
	\begin{align}\label{eq:taylor}
   		 	\Tcal_\Ncal(R_\ast(\Omega)) = \Ncal \cap \Ncal R_{\theta_n}^T \cap \Ncal R_{\theta_{n+1}}^T,
    \end{align}
 where $n\in\{1,...,N\}$ is uniquely determined by the relation $\theta_n < \frac{\pi}{2}\leq\theta_{n+1}$.
\end{proposition}
Considering that polycrystals generally consist of large number of grains with different orientations, this result simplifies the computation of this bound considerably, facilitating even an explicit analytical representation as discussed in Remark~\ref{rem:taylor}\,$e)$. 
In particular, we identify necessary and sufficient conditions on the slip directions such that $\Tcal_\Ncal(R_\ast(\Omega))$ is trivial, i.e., identical to the set of rotations $\SO(2)$, see Corollary \ref{cor:trivial_taylorbound}.
\smallskip

A key ingredient for deriving outer bounds on the attainable macroscopic strains is the generalization of the classical Hadamard jump condition~\cite{BaJ87} formulated in Theorem~\ref{theo:hadamard_curved}.
This tool puts us in the position to derive an outer bound on $\Fcal_\Mcal(\Omega,R_\ast)$ by analyzing the rank-one compatibility between the macroscopic and microscopic strains at the boundary grains of the polycrystal.
In particular, if the outer unit normal of $\Omega$ is, in a point, perpendicular to the slip orientation $s_i$ of the associated boundary grain $\Omega_i$, then 
$\Fcal_\Mcal(\Omega,R_\ast)$ is contained in $\Ncal_{s_i}$.
The following statement is a simplified version of Proposition \ref{prop:taylor_bdr} below. 
\begin{proposition}\label{prop:taylor_bdr_intro}
	Let $\Omega_1,\ldots,\Omega_M$ with $M\in \N$ be the boundary grains of $(\Omega,R_\ast)$ and let $J\subset \{1,\ldots,M\}$ be the set of all indices $i$ such that there exists $x_i\in \partial\Omega_i\cap\bdpoly$ with $\nu(x_i)\cdot R_{\ast}\restrict{\Omega_i}e_1=0$. 
	Then, 
	\begin{align}\label{outerbound_intro}
		\Fcal_\Mcal(\Omega,R_\ast)\subset \bigcap_{i\in J} \Ncal R_{\ast}^T\restrict{\Omega_i}.
	\end{align}
\end{proposition}

Observe that the outer bound in~\eqref{outerbound_intro} has the same overall structure as the aforementioned inner bound $\Tcal_\Ncal(\Omega;R_\ast)$ and can thus be simplified and expressed in the same way, cf.~Proposition \ref{prop:taylorbound}. 
In particular, Proposition \ref{prop:taylor_bdr_intro} allows us to conclude for examples of polycrystals with sufficient symmetry and selected bicrystals (see Examples~\ref{ex:symmetric_outer} and~\ref{ex:bicrystals}) that the attainable macroscopic strains coincide with the globally affine solutions of \eqref{relaxed}, i.e., $\Fcal_\Mcal(\Omega,R_\ast) = \Tcal_\Ncal(R_\ast(\Omega))$.

The Taylor bound for~\eqref{relaxed}, however, is not always optimal, meaning that the second inclusion in \eqref{taylor_inclusions} is strict in general. 
To see this, we employ a geometric setup similar to the construction in \cite{CKZ17, Pom10} and design a polycrystal together with a continuous and finitely piecewise affine solution to the relaxed problem \eqref{relaxed} with boundary values outside of the Taylor bound, which then gives rise to a Lipschitz solution to \eqref{unrelaxed} via Proposition \ref{prop:affine_lipschitz}.

\smallskip
This article is outlined as follows. 
A few preliminaries and mathematical tools for the analysis of the inclusions \eqref{unrelaxed} and \eqref{relaxed} and the corresponding sets $\Mcal_s$ and $\Ncal_s$ are collected in Section \ref{sec:preliminaries}.  
The focus of Section \ref{sec:taylor} is the characterization and discussion of the inner Taylor bounds. 
Outer bounds resulting from rank-one compatibility conditions at the boundary grains are derived in Section \ref{sec:outer_bounds} and illustrated by a few examples. Finally, we address the question of optimality of the Taylor bound for~\eqref{relaxed}, proving in Section~\ref{sec:constructions} that this is not the case in general.

\section{Preliminaries}\label{sec:preliminaries}

\subsection{Notation}\label{sec:notation}
Throughout this paper, we use the following notation.
The standard basis vectors in $\R^2$ are denoted by $e_1,e_2$, and $\Scal^1=\{x\in \R^2: |x|=1\}$ is the one-dimensional unit sphere with respect to the Euclidean norm.
We write $a\cdot b$ for the standard scalar product of two vectors $a, b\in \R^2$, and define their tensor product  $a\otimes b$ as $(a\otimes b)_{ij} = a_ib_j$ for $i,j\in\{1,2\}$. 
The space of real $2\times 2$ matrices is equipped with the Frobenius norm $\abs{\cdot}$. 
Moreover, $\Sl(2)$ consists of all matrices in $\R^{2\times 2}$ with determinant equal to one, and $\SO(2)$ denotes the set of rotations in $\R^{2\times 2}$; we write  
\begin{align*}
	R_\theta := \begin{pmatrix}\cos \theta & -\sin\theta\\\sin\theta&\cos\theta\end{pmatrix} \in\SO(2)
\end{align*}
for a rotation by an angle $\theta\in\R$, and define $a^\perp = R_{\frac{\pi}{2}}a$ for any $a\in \R^2$. 
For $x\in \R^2$, $\nu\in \Scal^1$, and $r>0$, let
\begin{align*}
	B^+_\nu(x,r) = \{y\in B(x,r) : (y-x)\cdot \nu > 0\}\quad\text{ and }\quad 
	B^-_\nu(x,r) = \{y\in B(x,r) : (y-x)\cdot \nu < 0\},
\end{align*}
where $B(x,r)\subset \R^2$ is the open ball with center $x\in \R^2$ and radius $r>0$.
If not specified otherwise, the two-dimensional Lebesgue-measure of a measurable set $U\subset \R^2$ is denoted by $|U|$.

The product and sum of two sets $\Fcal, \Gcal \subset\R^{2\times 2}$ are interpreted in the sense of Minkowski, i.e., $\Fcal \Gcal=\{FG: F\in \Fcal, G\in \Gcal\}$ and $\Fcal+\Gcal=\{F+G: F\in \Fcal, G\in \Gcal\}$. 
We work with the following definitions of (finite) generalized convex hulls: If $\Fcal\subset \R^{2\times 2}$, then
\begin{align}\label{qc_hull}
	\Fcal^\qc:= \bigl\{F\in \R^{2\times 2}: h(F) \leq \sup_{G\in \Fcal} h(G) \text{ for all } h: \R^{2\times 2} \to \R \text{ quasiconvex}\bigr\},
\end{align}
is the quasiconvex hull of $\Fcal$, and $\Fcal$ is called quasiconvex if $\Fcal=\Fcal^\qc$; the polyconvex and rank-one convex hulls $\Fcal^\pc$ and $\Fcal^\rc$ as well as polyconvexity and rank-one convexity of $\Fcal$ are introduced analogously. More details on the generalized notions of convexity for sets can be found in, e.g.,~\cite[Chapter~7]{Dac08}.

For a set $U\subset\R^2$, we define the indicator function $\ONE_U$ as $\ONE_U(x)=1$ for $x\in U$ and $\ONE_U(x)=0$ otherwise. 
Throughout, we adopt the standard notation for Lebesgue spaces, Sobolev spaces, spaces of functions of bounded variation, and spaces of continuously differentiable functions; particularly, $L^\infty(U;\R^2)$, $W^{1,\infty}(U;\R^2)$, $BV_{\rm loc}(U;\R^{2\times 2})$ and $C^1(\overline{U};\R^2)$ for open $U\subset \R^2$.

\subsection{Technical tools and auxiliary results}
We begin with the following connection between the differential inclusion problem \eqref{unrelaxed} and its relaxed version \eqref{relaxed}. 

\begin{proposition}\label{prop:affine_lipschitz}
	If there is a continuous, finitely piecewise affine solution to the relaxed problem
	\eqref{relaxed}, then there exists a Lipschitz solution to \eqref{unrelaxed} with the same boundary values.
\end{proposition}

This result is based on the analysis in the paper~\cite{CoT05}  by Conti \& Theil on a model in single-crystal plasticity with one active slip system. 
More precisely, it follows from \cite[Theorem 4]{CoT05} applied grain-wise to each of the affine components of the solution to~\eqref{relaxed}, along with \cite[Lemma 2]{CoT05}, which again relies on the convex integration theory by M\"uller \& \v{S}ver\'{a}k~\cite[Theorem 1.3]{MuS99}. 

As a consequence of Proposition~\ref{prop:affine_lipschitz}, one obtains an inner bound on $\Fcal_\Mcal(\Omega,R_\ast)$ by identifying globally affine solutions to \eqref{relaxed}, cf.~\eqref{taylor_sets} and Section~\ref{sec:taylor}.
\medskip

The subsequent auxiliaries on rank-one connectedness with the sets $\Ncal_s$ are used in Section \ref{sec:outer_bounds} in the context of outer bounds resulting from the rank-one compatibility at the boundary grains.
Considering that $\Sl(2)$ is a trivial outer bound for all polycrystals, it suffices to restrict the discussion to this class of matrices. 
We first introduce the following terminology.

\begin{definition}[\boldmath{$\nu$-compatibility}]\label{def:compatibility}
Let $\nu\in \Scal^1$, $A\in \R^{2\times 2}$, and $\Bcal\subset \R^{2\times 2}$.
We call $A$ $\nu$-compatible with $\Bcal$, if $A$ is rank-one connected to $\Bcal\subset \R^{2\times 2}$ along linear interfaces with normal $\nu$, i.e., if there exists $a\in \R^2$ such that $a\otimes \nu \in \Bcal - A$, or equivalently, if $A\nu^\perp\in \Bcal\nu^\perp$.
\end{definition}

The next lemma gives a simple condition for the rank-one connectedness between a given element in $\Sl(2)$ and $\Ncal_s$. 
Our calculations are based on the observation that, for any $s\in\Scal^1$, the set $\Ncal_s$ consists of matrices of the form 
\begin{align}\label{matrices_Ncal_s}
	R\bigl(\beta s\otimes s + \tfrac{1}{\beta}s^\perp\otimes s^\perp + \gamma s\otimes s^\perp\bigr)
\end{align}
with $R\in\SO(2)$, $\beta\in(0,1]$ and $\gamma\in \R$. Moreover, if $\beta>0$ then \eqref{matrices_Ncal_s} describes all elements in the larger set $\Sl(2)$.

\begin{lemma}\label{lem:rank-one}
	Let $s, \nu \in \Scal^1$ with $s\cdot \nu\neq 0$. 
	Further, let $F\in\Sl(2)$, which is assumed to be represented as in~\eqref{matrices_Ncal_s} with $R\in \SO(2)$, $\beta>0$, and $\gamma\in \R$. Then the following statements are equivalent:
	\begin{itemize}
		\item[$(i)$] $F$ is $\nu$-compatible with $\Ncal_s$;
		\item[$(ii)$] $F$ is $\nu$-compatible with $\Mcal_s$;
		\item[$(iii)$] it holds that
			\begin{align}\label{eq:rank-one}
				\Bigl(\frac{s\cdot \nu^\perp}{s\cdot \nu}\beta +\gamma\Bigr)^2 + \frac{1}{\beta^2}\geq 1.
			\end{align}
	\end{itemize}
\end{lemma}
\begin{proof}
	We start with a useful general equivalence:
	Any two matrices $G,\bar{G}\in \Sl(2)$ of the form \eqref{matrices_Ncal_s} with $R=\Id,\bar{R}\in \SO(2)$, $\beta,\bar{\beta}>0$, $\gamma,\bar{\gamma}\in \R$, respectively, satisfy 
	\begin{align}\label{Fnuperp}
		G\nu^\perp = \bar{G}\nu^\perp
	\end{align}		
	if and only if
	\begin{numcases}{}
		\bar\beta\frac{s\cdot \nu^\perp}{s\cdot \nu} + \bar\gamma = N_{s,\nu}(\beta,\gamma)\cdot \bar{R}s,\label{e1}\\
		\frac{1}{\bar\beta} = \bar Rs^\perp\cdot N_{s,\nu}(\beta,\gamma),\label{e2}
	\end{numcases}
	with
	\begin{align*}
		N_{s,\nu}(\beta,\gamma)=\Big(\frac{s\cdot \nu^\perp}{s\cdot \nu}\beta + \gamma\Big)s + \frac{1}{\beta}s^\perp;
	\end{align*}
	this follows simply via scalar multiplication of \eqref{Fnuperp} with $\bar{R}s$ and $\bar{R}s^\perp$. 
	\medskip	
	
	To show the implication $(i)\Rightarrow (iii)$, let $F$ be as in the statement and assume without restriction that $R=\Id$. 
	Suppose that there exists $\bar{F}\in \Ncal_s$ of the form \eqref{matrices_Ncal_s} with $\bar\beta\in(0,1]$, $\bar\gamma \in \R$, and $\bar R\in\SO(2)$ such that \eqref{e1} and \eqref{e2} hold. 	
	Estimating the right-hand side of \eqref{e2} yields
	\begin{align*}
		1\leq \frac{1}{\bar\beta^2} \leq |N_{s,\nu}(\beta,\gamma)|^2 = \left(\frac{s\cdot \nu^\perp}{s\cdot \nu}\beta + \gamma\right)^2 + \frac{1}{\beta^2},
	\end{align*}
	which is the desired inequality \eqref{eq:rank-one}. 
	
	As for $(iii)\Rightarrow (ii)$, assume that~\eqref{eq:rank-one} is fulfilled and take $\bar\beta=1$.
	It remains to find $\bar{\gamma}\in\R$ and $\bar{R}\in\SO(2)$ such that \eqref{e1} and \eqref{e2} are satisfied.
	For $\bar{R}$, it is enough to determine $\bar{R}s^\perp$, which follows from considering solutions $\xi\in\Scal^1$ to
	\begin{align*}
		1 = \xi \cdot N_{s,\nu}(\beta,\gamma);
	\end{align*}
	these exist since $|N_{s,\nu}(\beta,\gamma)|\geq 1$ due to \eqref{eq:rank-one}.
	Therefore, \eqref{e2} is true and $\bar\gamma$ is uniquely determined by $\bar R$, $\bar{\beta}$ and \eqref{e1}. 
	Since the implication $(ii)\Rightarrow (i)$ is trivial, the statement is proven. 
\end{proof}

The next remark addresses the case of $s^\perp$-compatibility with the sets $\Ncal_s$ and $\Mcal_s$.

\begin{remark}\label{rem:rank_one}
	Let $s\in \Scal^1$ and $F\in \Sl(2)$.
	
	$a)$ Basic geometric arguments show:
	\begin{center} 
		$F$ is $s^\perp$-compatible with $\Ncal_s$ if and only if $F\in\Ncal_s$;
	\end{center}
 	\begin{center}
 		$F$ is $s^\perp$-compatible with $\Mcal_s$ if and only if $F\in\Mcal_s$.
	\end{center}
	Observe in particular, that the equivalence $(i)\Leftrightarrow (ii)$ in Lemma \ref{lem:rank-one} is not valid in this case. 
	\medskip
	
	$b)$ Let $\Dcal\subset \Scal^1$ be dense in $\Scal^1$. It holds that $F$ is $s^\perp$-compatible with $\Ncal_s$ if and only if $F$ is $\nu$-compatible with $\Ncal_s$ for all  $\nu\in \Dcal$. 
	Indeed, assuming that $\pm s^\perp\notin \Dcal$, that is a consequence of $(i)\Leftrightarrow (iii)$ in Lemma \ref{lem:rank-one} since
	\begin{align*}
		t\mapsto (t\beta + \gamma)^2 + \frac{1}{\beta^2}\geq 1 \quad\text{on a dense subset of $\R$}
	\end{align*}
	if and only if $|Fs|= \beta \leq  1$, which is equivalent to the $s^\perp$-compatibility of $F$ with $\Ncal_s$ due to $a)$. 
\end{remark}

The last auxiliary result follows from a slight modification of the first step in the proof of~\cite[Theorem 1.1]{CDK13a}, according to which the generalized convex hulls of $\Mcal_{e_1}\cup \Mcal_{e_2}$ coincide with $\Sl(2)$.  
Instead of the two orthogonal slip directions $e_1$ and $e_2$, we consider here two linearly independent orientations $s, s'$. 
Alternatively, the next proposition can be seen as a consequence of the sharper result \cite[Lemma 4.17]{Sch13}, where the lamination convex hull of $\Mcal_{s}\cup\Mcal_{s'}$ is identified as $\Sl(2)$, using first-order laminates.  
For the reader's convenience, we include the simpler proof based on the strategy in \cite{CDK13a}.

\begin{proposition}\label{prop:qc_hull}
	Let $s,s'\in\Scal^1$ with $s'\neq \pm s$. Then,
	\begin{align*}
		(\Mcal_s \cup \Mcal_{s'})^\rc=(\Mcal_s \cup \Mcal_{s'})^\qc=(\Mcal_s \cup \Mcal_{s'})^\pc = \Sl(2).
	\end{align*}
\end{proposition}
\begin{proof}
	In light of
	\begin{align*}
		(\Ncal_{s}\cup\Ncal_{s'})^\rc = (\Mcal_s^\rc \cup \Mcal_{s'}^\rc)^\rc = (\Mcal_s \cup \Mcal_{s'})^\rc \subset (\Mcal_s \cup \Mcal_{s'})^\pc  \subset \Sl(2),
	\end{align*}
	cf.~\eqref{Ns}, it suffices to prove  that $\Sl(2)\subset (\Ncal_s \cup \Ncal_{s'})^\rc$. 
	To this end, we show that any $F\in\Sl(2)$ can be expressed as a convex combination of rank-one connected matrices $F_+,F_-\in\Ncal_s\cup\Ncal_{s'}$, and exploit that the lamination convex hull is contained in the rank-one convex hull \cite[Theorems 7.17 and 7.28]{Dac08}. 

	Suppose without loss of generality that $F\in \Sl(2)\setminus (\Ncal_{s}\cup\Ncal_{s'})$ and consider the rank-one line
	\begin{align*}
		t\mapsto F_t = F(\Id + t(s+s')\otimes (s-s')),
	\end{align*}
 	along which the determinant is constant by construction; also, assume that $Fs\cdot Fs'\leq 0$, otherwise switch the roles of $s+s'$ and $s-s'$ in the rank-one line. 
	Then it holds that
	\begin{align*}
		|F(s+s')|^2 = |Fs|^2 + 2Fs\cdot Fs' + |Fs'|^2 \leq |F(s-s')|^2,
	\end{align*}
	with equality if and only if $Fs\cdot Fs'=0$. 
	If the latter is satisfied, the estimate 
	\begin{align*}
		\det(Fs|Fs') = \det F\det(s|s') = \det (s|s') \leq 1,
	\end{align*}
	implies that $\min\{|Fs|,|Fs'|\}\leq 1$.
	However, this contradicts $F\in \Sl(2)\setminus (\Ncal_{s}\cup\Ncal_{s'})$, which is why we take $|F(s+s')| < |F(s-s')|$ in the following.
	
The quadratic function
	\begin{align*}
		\ffi: \R \to \R,\quad t\mapsto |F_t(s-s')|^2 - |F_t(s+s')|^2
	\end{align*}
	then satisfies $\ffi(0) <0$ and $\ffi''>0$ pointwise. 
	The latter is a direct consequence of $|F_t(s+s')| = |F(s+s')|$ for every $t\in\R$, utilizing that $(s+s')\cdot(s-s')=0$.
	Hence, there exist $t_-<0$ and $t_+>0$ with $\ffi(t_\pm) =0$, so that $F_-:= F_{t_-}$ and $F_+:=F_{t_+}$ satisfy $|F_\pm(s+s')| = |F_\pm(s-s')|$. 
	Exactly as before, one concludes that $\min\{|F_\pm s|,|F_\pm s'|\} \leq 1$,
	and thus, $F_\pm \in \Ncal_s \cup \Ncal_{s'}$, as desired.
\end{proof}

The previous result is used in Section~\ref{sec:outer_bounds} to explain why considering boundary dual points, cf.~\eqref{bddual_def}, does not add any non-trivial contributions to outer bounds emerging from rank-one compatibility at the boundary grains.

\section{A geometry-independent inner bound}\label{sec:taylor}
In this section, we identify and characterize inner bounds on $\Fcal_\Mcal(\Omega,R_\ast)$ in terms of the globally affine solutions to \eqref{unrelaxed} and \eqref{relaxed}.
Note that such solutions merely take the occurring slip directions, but not the geometry of the grains, i.e., their size and shape, into account. 
The expressions to be determined are the Taylor bounds $\Tcal_\Mcal(R_\ast(\Omega))$ and $\Tcal_\Ncal(R_\ast(\Omega))$ given by the finite intersections in \eqref{taylor_sets},
which, in light of Proposition \ref{prop:affine_lipschitz}, satisfy the relation
\begin{align*}
	\Tcal_\Mcal(R_\ast(\Omega)) \subset \Tcal_\Ncal(R_\ast(\Omega)) \subset \Fcal_\Mcal(\Omega,R_\ast).
\end{align*}
First, we prove Proposition \ref{prop:taylorbound}, showing that $\Tcal_\Ncal(R_\ast(\Omega))$ depends on at most three slip orientations, no matter the number of grains of the polycrystal $(\Omega, R_\ast)$; the same statement is valid for $\Tcal_\Mcal(R_\ast(\Omega))$, see Remark \ref{rem:taylor}\,$f)$ below. 

\begin{proof}[Proof of Proposition \ref{prop:taylorbound}]
	We first establish an explicit expression of the intersection $\Ncal \cap \Ncal R_\theta^T$ with $\theta\in (0, \pi)$.
	In light of \eqref{matrices_Ncal_s} for $s=e_1$, it holds that 
	$F\in\Ncal$ if and only if $F=S(\beta e_1|\tfrac{1}{\beta}e_2 + \gamma e_1)$ for some $S\in\SO(2),\beta\in(0,1],$ and $\gamma\in\R$.
	In order to characterize $\beta$ and $\gamma$ such that $FR_{\theta}\in \Ncal$, we observe that $\det(FR_{\theta}) = \det F = 1$ and that
	the constraint $|FR_{\theta} e_1|\leq 1$ can be rewritten as
	\begin{align}\label{quadratic_inequality}
		0\geq |FR_{\theta}e_1|^2 - 1 =\gamma^2\sin^2\theta + 2\gamma\beta \cos\theta\sin\theta +\frac{1}{\beta^2}\sin^2\theta + \beta^2 \cos^2\theta - 1.
	\end{align}
	As the right-hand side is quadratic in $\gamma$ with positive leading coefficient, it suffices to determine its zeroes to solve the inequality \eqref{quadratic_inequality}. 
	In doing so, we find that
	\begin{align}\label{image_compact}
		\Ncal \cap \Ncal R^T_\theta = \SO(2)\psi(\Lambda_\theta\big)
	\end{align}
	with $\psi : (0,1]\times \R\to \Ncal$ given by $(\beta,\gamma) \mapsto \big(\beta e_1|\tfrac{1}{\beta} e_2 + \gamma e_1 \big)$ and
	\begin{align}\label{plane_representation}
		\Lambda_\theta=\big\{(\beta,\gamma)\in \R^2 : \beta \in [\sin \theta,1], \gamma \in \Gamma(\theta, \beta)\big\},
	\end{align}
	where
	\begin{align}\label{def_interval}
    	\Gamma(\theta,\beta) = [\gamma_-(\theta,\beta),\gamma_+(\theta,\beta)]\subset \R \quad \text{with $\gamma_\pm(\theta,\beta) = -\beta\cot\theta \pm \sqrt{(\sin \theta)^{-2}-\beta^{-2}}$}
    \end{align}
    for $\theta\in (0,\pi)$ and $\beta\in [\sin \theta, 1]$.
    Moreover, note that $(1, 0)\in \Lambda_\theta$ for any $\theta\in (0, \pi)$ and that 
	\begin{align*}
		\Lambda_{\theta}=\{(1, 0)\} \quad \text{if and only if} \quad \theta=\frac{\pi}{2}.
	\end{align*}
	In light of \eqref{image_compact} and the injectivity of $\psi$, one immediately obtains through iteration that
	\begin{align}\label{intersections}
		\Tcal_\Ncal(R_\ast(\Omega)) = \Ncal \cap \Ncal R_{\theta_2}^T\cap\ldots\cap \Ncal R_{\theta_N}^T =  \SO(2)\psi(\Lambda_{\theta_2}\cap\ldots \cap \Lambda_{\theta_N}),
	\end{align}
	recalling that $\theta_1=0$. 
	
	In the following, we discuss the cases $n=N, n=1,$ and $n\notin\{1,N\}$ separately.  
	\medskip

	\textit{Step 1: The case $n=N$.} 
	The essence of this step is the observation that the sets $\Lambda_{\theta_i}$ for $i=1, \ldots, N$ are strictly decreasing nested sets, that is,
	\begin{align}\label{Acal_inclusions}
		\Lambda_{\theta_2} \supsetneq \ldots \supsetneq \Lambda_{\theta_N};
	\end{align}
	the representation \eqref{intersections} then simplifies to $\Tcal_\Ncal(R_\ast(\Omega))=\Ncal \cap \Ncal R_{\theta_N}^T$, which results in~\eqref{eq:taylor}, given that $\Ncal R_\pi^T = \Ncal$, cf.~also Remark~\ref{rem:taylor}\,$c)$.
	Since the sine function is strictly increasing in $(0,\frac{\pi}{2})$, it suffices for~\eqref{Acal_inclusions} to prove that
	\begin{align}\label{Gamma_inclusions}
		\Gamma(\theta,\beta)\subsetneq \Gamma(\tilde\theta,\beta)\quad  \text{ for any $\theta,\tilde \theta\in (0,\tfrac{\pi}{2})$ with $\tilde \theta<\theta$ and $\beta\in[\sin\theta,1]$.}
	\end{align}

	To this end, we show that $\partial_\theta \gamma_+(\theta,\beta) <0$ and $\partial_\theta \gamma_-(\theta,\beta) >0$  
	for all $\theta\in (0,\frac{\pi}{2})$ and $\beta\in(\sin\theta, 1)$, cf.~\eqref{def_interval}; the case $\beta=1$ is treated separately below.
A direct calculation yields that 
	\begin{align*}
		\partial_\theta \gamma_\pm (\theta,\beta) &= \partial_\theta\left(\beta\cot\theta \pm \sqrt{(\sin\theta)^{-2}-\beta^{-2}}\right)\\
			&= -\frac{1}{\sin^2\theta}\left(\beta \mp \frac{\cos\theta}{\sqrt{1-\beta^{-2}\sin^2\theta}}\right)\
			 = \frac{1}{\sin^2\theta\sqrt{1-\beta^{-2}\sin^2\theta}}\left(\sqrt{\beta^2-\sin^2\theta}\mp\cos \theta\right),
	\end{align*}
	where the first factor is clearly positive; the second factor is strictly increasing in $\beta$ on $(\sin\theta,1]$ with boundary values 
	\begin{align*}
		\sqrt{\beta^2-\sin^2\theta}\mp\cos \theta = \begin{cases}
			\mp \cos\theta &\text{ if } \beta = \sin\theta,\\
			\cos\theta\mp \cos\theta &\text{ if } \beta = 1,
		\end{cases}
	\end{align*}
	which implies the desired monotonicity of $\gamma_+(\theta,\beta)$ and $\gamma_-(\theta,\beta)$ in $\theta$.
	
	For $\beta=1$, we compute that
	\begin{align}\label{trivial_interval}
		\Gamma(\theta, 1) = \begin{cases}
			[-2\cot\theta,0] &\text{ if } \theta\in (0,\frac{\pi}{2}),\\
			[0,-2\cot\theta] &\text{ if } \theta \in [\frac{\pi}{2},\pi),
			\end{cases}
	\end{align}
	from which \eqref{Gamma_inclusions} follows immediately in that case.
	\medskip
	
	\textit{Step 2: The case $n=1$.}  Now, let $\frac{\pi}{2}\leq\theta_2< \ldots < \theta_N.$
	In this case, formula \eqref{intersections} reduces to $\Tcal_\Ncal(R_\ast(\Omega)) = \Ncal\cap \Ncal R_{\theta_2}^T$ since the chain of inclusions in \eqref{Acal_inclusions} is reversed. 
	The latter is due to the monotonicity of the sine function in $[\frac{\pi}{2},\pi)$, as well as the
	symmetry of the upper and lower bounds in \eqref{def_interval} around $\frac{\pi}{2}$
	in the sense that
	\begin{align}\label{symmetric_angles}
		\gamma_\pm(\tfrac{\pi}{2}+\theta,\beta) = - \gamma_\mp(\tfrac{\pi}{2}-\theta,\beta)
	\end{align}
	for $\theta\in(0,\frac{\pi}{2})$ and $\beta\in [\sin(\frac{\pi}{2}-\theta),1]$. 
	Precisely, combining \eqref{symmetric_angles} with the arguments of Step 1 yields that 
	\begin{align*}
		\Gamma(\theta,\beta) \subsetneq \Gamma(\tilde\theta,\beta)\quad  \text{ for any $\theta,\tilde\theta\in (\tfrac{\pi}{2},\pi)$ with $\theta < \tilde\theta$ and $\beta\in[\sin\theta,1]$.}
	\end{align*}
	\medskip
			
	\textit{Step 3: The case $n\notin\{1,N\}$.} Applying the conclusion of Step 1 to $\Ncal \cap\ldots \cap \Ncal R_{\theta_n}^T$  and the results of Step 2 to $\Ncal \cap \Ncal R_{\theta_{n+1}} \cap \ldots \cap \Ncal R_{\theta_N}$ gives
	\begin{align*}
		\Tcal_\Ncal(R_\ast(\Omega)) = \Ncal \cap \Ncal R_{\theta_2}^T\cap\ldots\cap \Ncal R_{\theta_n}^T \cap \Ncal R_{\theta_{n+1}}^T \cap \ldots \cap \Ncal R_{\theta_N}^T = \Ncal \cap \Ncal R_{\theta_n}^T \cap \Ncal R_{\theta_{n+1}}^T,
	\end{align*}
	as stated.
\end{proof}

\begin{remark}\label{rem:taylor}
Let $(\Omega,R_\ast)$ be as in Proposition \ref{prop:taylorbound} and $n\in\{1,\ldots,N\}$ with $\theta_n<\frac{\pi}{2}\leq \theta_{n+1}$.\medskip

	$a)$ Observe that $\SO(2)\subset \Tcal_\Ncal(R_\ast(\Omega))$ for any polycrystal $(\Omega,R_\ast)$. 
	Therefore, we call the Taylor bound $\Tcal_\Ncal(R_\ast(\Omega))$ trivial if $\Tcal_\Ncal(R_\ast(\Omega)) =\SO(2)$.\medskip	
	
	$b)$ The Taylor bound $\Tcal_\Ncal(R_\ast(\Omega))$ is polyconvex as the intersection of polyconvex hulls, and it is compact whenever $(\Omega, R_\ast)$ is not a single-crystal. One way to see this is via the representation formula \eqref{image_compact}-\eqref{def_interval}, where $\Tcal_\Ncal(R_\ast(\Omega))$ is expressed as the image of a compact set under a continuous map. 
	\medskip
	
	$c)$ If $\theta_k \leq \frac{\pi}{2}$ (or $\theta_k \geq \frac{\pi}{2}$) for all $k\in\{2,\ldots,N\}$, then only two values of $R_\ast$, precisely $R_{\theta_1}$ and $R_{\theta_N}$ (or $R_{\theta_1}$ and $R_{\theta_2}$), are sufficient for characterizing the Taylor bound, which follows directly from	Steps 1 and 2 in the proof of Proposition \ref{prop:taylorbound}. 
	This  observation  is in agreement with \eqref{eq:taylor} since $\Ncal R_{\theta_1}^T = \Ncal = \Ncal R_{\pi}^T = \Ncal R_{\theta_{N+1}}^T$.	
	\medskip
	
	$d)$ Proposition~\ref{prop:taylorbound} shows that the Taylor bound depends on at most three different slip orien\-tations. 
	Indeed, it involves the slip direction $s=e_1$, corresponding to $\theta_1=0$, and at most two others that are closest to $e_2$, see Figure \ref{fig:angle_choice}.
	For a more general setting without the restriction $\theta_1=0$, we refer to Remark \ref{rem:first_angle}.
	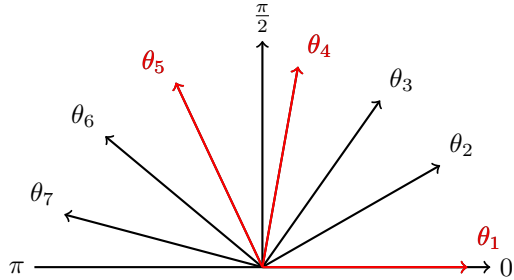
\begin{figure}[h!]
		\centering
		\begin{tikzpicture}
			\begin{scope}[scale=0.6]
				\draw [thick, ->] (-5,0)--(5,0);
				\draw (-5,0) node [anchor = east]{$\pi$};
				\draw (5,0) node [anchor = west]{$0$};
				\draw [thick, ->] (0,0)--(0,5);
				\draw (0,5) node [anchor = south]{$\frac{\pi}{2}$};
				
				\draw [->,thick](0,0) -- ++(0:4.5);
				\draw ($(0,0.2) + (0:4.5)$) node [anchor = south west] {$\theta_1$};
				\draw [thick, ->] (0,0) --++(30:4.5);
				\draw ($(0,0) + (30:4.5)$) node [anchor = south west] {$\theta_2$};
				\draw [thick, ->] (0,0) --++(55:4.5);
				\draw ($(0,0) + (55:4.5)$) node [anchor = south west] {$\theta_3$};
				\draw [->,thick] (0,0) --++(80:4.5);
				\draw ($(0,0) + (80:4.5)$) node [anchor = south west] {$\theta_4$};
				\draw [->,thick] (0,0) --++(115:4.5);
				\draw ($(0,0) + (115:4.5)$) node [anchor = south east] {$\theta_5$};
				\draw [thick, ->] (0,0) --++(140:4.5);
				\draw ($(0,0) + (140:4.5)$) node [anchor = south east] {$\theta_6$};
				\draw [thick, ->] (0,0) --++(165:4.5);
				\draw ($(0,0) + (165:4.5)$) node [anchor = south east] {$\theta_7$};

				\draw [->,thick,red](0,0) -- ++(0:4.5);
				\draw [->,thick,red] (0,0) --++(80:4.5);
				\draw [->,thick, red] (0,0) --++(115:4.5);
				\draw ($(0,0) + (80:4.5)$) node [anchor = south west, red] {$\theta_4$};	
				\draw ($(0,0.2) + (0:4.5)$) node [anchor = south west, red] {$\theta_1$};
				\draw ($(0,0) + (115:4.5)$) node [anchor = south east, red] {$\theta_5$};
			\end{scope}
		\end{tikzpicture}
		\caption{Dependence of the Taylor bound on exactly three slip directions (marked in red); 
		here, $n=4$ and $\Tcal_\Ncal(\{ R_{\theta_1},\ldots,R_{\theta_7}\}) = \Tcal_\Ncal(\{R_{\theta_1},R_{\theta_4},R_{\theta_5}\})$.}\label{fig:angle_choice}
	\end{figure}

	$e)$ Let $F\in \Ncal$ be represented by $F=S(\beta e_1|\tfrac{1}{\beta}e_2 + \gamma e_1)$ with $S\in\SO(2), \beta\in(0,1],$ and $\gamma\in\R$.
	As a consequence of the previous proof, one can extract the following useful equivalence for explicit calculations: In fact,  $F\in \Tcal_\Ncal(R_\ast(\Omega))$ if and only if
	\begin{align*}
		(\beta, \gamma)\in \Lambda_{\theta_{n}}\cap \Lambda_{\theta_{n+1}}
	\end{align*}
	or equivalently, 
	\begin{align*}
		 \beta\in [\max\{\sin \theta_n, \sin \theta_{n+1}\}, 1] \text{ and } \gamma\in \Gamma(\theta_n, \beta)\cap \Gamma(\theta_{n+1}, \beta),
	\end{align*}
	see also \eqref{plane_representation} and \eqref{def_interval}. 
	The sets $\Lambda_\theta$ for $\theta\in (0,\pi)$ (as well es their intersections) can be illustrated as in Figure \ref{fig:taylor}; in particular, this figure depicts both the properties that $\Lambda_{\frac{\pi}{2} +\theta}$ emerges from $\Lambda_{\frac{\pi}{2} - \theta}$ for $\theta\in(0,\frac{\pi}{2})$ via reflection (see the green and blue areas), as well as the nested structure in \eqref{Acal_inclusions} (see red and blue). 
	\medskip
	
	$f)$ Regarding the Taylor bound for the unrelaxed problem \eqref{unrelaxed}, we find that
	 $\Tcal_\Mcal(R_\ast(\Omega))$ is compact for any non-trivial polycrystal and that 
	\begin{align*}
		\Tcal_\Mcal(R_\ast(\Omega)) =\SO(2)
	\end{align*}
	if and only if $n\notin \{1, N\}$. It is evident from the calculations in the proof of Proposition \ref{prop:taylorbound} that
	\begin{align}\label{taylor_Mcal}
		\Mcal \cap \Mcal R_\theta^T = \SO(2)\psi\bigl(\{1\}\times\Gamma(1,\theta)\bigr)
	\end{align}
	for any $\theta\in (0,\pi)$, where $\Gamma(1,\theta)$ is given in \eqref{trivial_interval}. The identity \eqref{taylor_Mcal} can be visualized by considering the ``right boundary'' $\Lambda_{\theta}\cap(\{1\}\times \R)$ of $\Lambda_{\theta}$ in Figure \ref{fig:taylor}.
	\begin{figure}
		\includegraphics[scale=1]{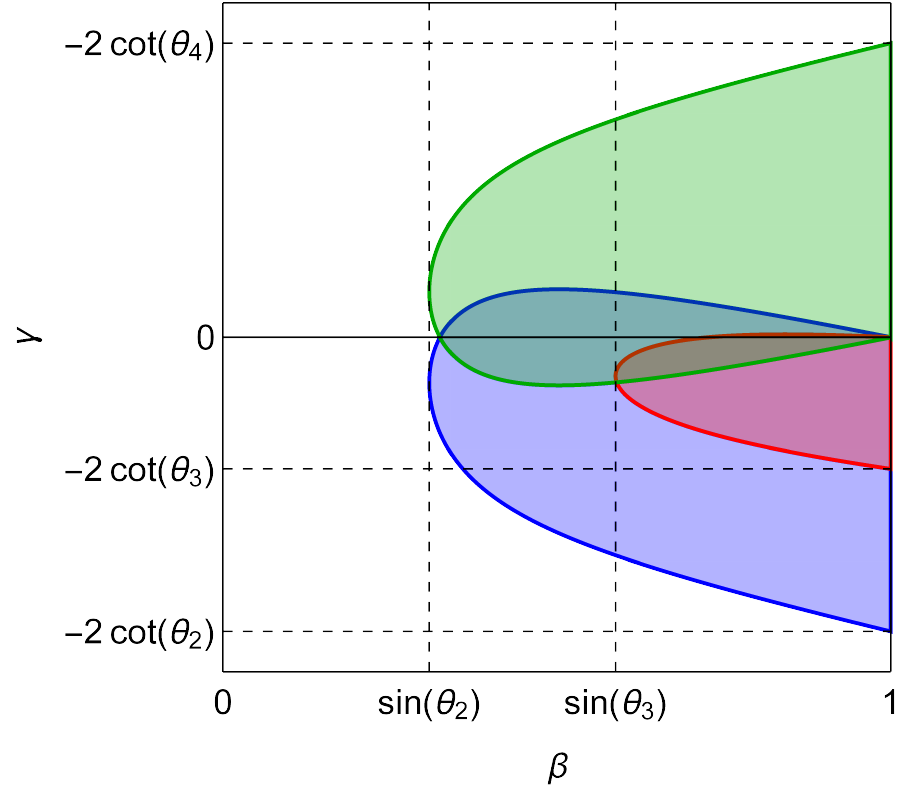}
		\caption{Illustration of sets $\Lambda_\theta$ for some choices of $\theta$. Here, the blue, red, and green areas describe $\Lambda_{\theta_2}$, $\Lambda_{\theta_3}$, and $\Lambda_{\theta_4}$,
		with $\theta_2=\frac{\pi}{10}, \theta_3 = \frac{2\pi}{10},$ and $\theta_4=\frac{9\pi}{10}$, respectively.}\label{fig:taylor}
	\end{figure}		
\end{remark}

Based on Proposition \ref{prop:taylorbound}, we now provide a necessary and sufficient condition on the orientations of a polycrystal $(\Omega, R_\ast)$ such that $\Tcal_\Ncal(R_\ast(\Omega))$ is trivial.
\begin{corollary}[Trivial Taylor bound]\label{cor:trivial_taylorbound}
	Let $(\Omega,R_\ast)$ be as in Proposition~\ref{prop:taylorbound}. 
	Then, the Taylor bound $\Tcal_\Ncal(R_\ast(\Omega))$ is trivial, i.e., $\Tcal_\Ncal(R_\ast(\Omega)) =\SO(2)$, if and only if there exists $n\in\{1,\ldots,N-1\}$ such that
	\begin{align*}
		\frac{\pi}{2} \in [\theta_n, \theta_{n+1}]\quad \text{and}\quad  \theta_{n+1}-\theta_n\leq \frac{\pi}{2}.
	\end{align*}
\end{corollary}

\begin{proof}
	\textit{Step 1: Sufficiency.} 
	If $\theta_{n}=\frac{\pi}{2}$ for some $n\in\{2, \ldots, N-1\}$, then any $F\in\Tcal_\Ncal(R_\ast(\Omega))$ satisfies
	\begin{align*}
		\det F=1,\quad |Fe_1|\leq 1,\quad |Fe_2|=|FR_{\theta_{n}}e_1|\leq 1. 
	\end{align*}	 
	Therefore, $F$ must be a rotation. Indeed, we estimate
	\begin{align*}
		1 = \det F = (Fe_1)^\perp\cdot(Fe_2)\leq |Fe_1||Fe_2|\leq 1,
	\end{align*}
	which implies that $|Fe_1| = |Fe_2| = 1$; in combination with $\det F=1$, this implies $F\in\SO(2)$.
	The same arguments can be used to treat the case $\theta_{n+1}=\frac{\pi}{2}$ for some $n\in \{1, \ldots, N-1\}$.

	It remains to show the claim for the case when
	$0<\theta_n < \frac{\pi}{2} < \theta_{n+1}  \leq \theta_n + \frac{\pi}{2}<\pi$ for $n\in \{2, \ldots, N-1\}$. 
	We apply Proposition~\ref{prop:taylorbound} twice, once to $(\Omega,R_\ast)$ and then again to another polycrystal $(\Omega, \tilde R_\ast)$ with $\tilde R_\ast(\Omega)= \{R_{\theta_1}, R_{\theta_n}, R_{\theta_{n+1}}, R_{\theta_n +\tfrac{\pi}{2}}\}$ to conclude that
	\begin{align*}
		\Tcal_\Ncal(R_\ast(\Omega)) = \Ncal\cap \Ncal R_{\theta_n}^T\cap \Ncal R_{\theta_{n+1}}^T = \Ncal\cap \Ncal R_{\theta_n}^T\cap \Ncal R_{\theta_{n+1}}^T \cap \Ncal R_{\theta_{n}+\frac{\pi}{2}}^T = \Tcal_\Ncal(\tilde R_\ast(\Omega)).
	\end{align*}
	The right-hand side is trivial since $\Ncal R_{\theta_n}^T\cap \Ncal R_{\theta_{n}+\frac{\pi}{2}}^T= (\Ncal \cap \Ncal R_{\frac{\pi}{2}}^T)R_{\theta_n}^T=\SO(2)$, which yields the statement.  \medskip

	\textit{Step 2: Necessity.} 
	Arguing by contraposition, assume first that either $\theta_k<\frac{\pi}{2}$ or $\theta_k> \frac{\pi}{2}$ for all $k\in\{2,\ldots, N\}$. 
	These two scenarios correspond exactly to those of Steps~1 and~2 in the proof of Proposition \ref{prop:taylorbound}, where it is shown that 
	\begin{align*}
		\Tcal_{\Ncal}(R_\ast(\Omega))
			= \SO(2)\psi(\Lambda_{\theta_N}) \quad\text{or}\quad
		\Tcal_{\Ncal}(R_\ast(\Omega))
			= \SO(2)\psi(\Lambda_{\theta_2}),
	\end{align*}
	respectively, cf.~\eqref{plane_representation}.  
	Since the sets $\Lambda_{\theta_N}$ and $\Lambda_{\theta_2}$ contain strictly more elements than $(1, 0)$ due to $\theta_2, \theta_N\neq \frac{\pi}{2}$ and since $\psi$ is injective
	with $\psi(\beta, \gamma)\in \SO(2)$ exactly for $(\beta, \gamma)=(1,0)$, it follows that $\SO(2)\subsetneq \Tcal_\Ncal(R_\ast(\Omega))$.
	
	Now it remains to address the case when there exists a $k\in \{2, \ldots, N-1\}$ such that $\theta_{k+1}-\theta_k>\frac{\pi}{2}$, or equivalently, 
	\begin{align}\label{angles_perp}
		0<\theta_k<\frac{\pi}{2}< \theta_k + \frac{\pi}{2} < \theta_{k+1}<\pi.
	\end{align}
	To show that $\Tcal_\Ncal(R_\ast(\Omega)) = \SO(2)\psi(\Lambda_{\theta_k}\cap \Lambda_{\theta_{k+1}})$ is strictly larger than the set of rotations, we show that there exists a $\beta\in [\max\{\sin\theta_k,\sin \theta_{k+1}\}, 1)$ such that
	\begin{align}\label{intersection}
		\Gamma(\theta_k,\beta)\cap\Gamma(\theta_{k+1},\beta)\neq \emptyset,
	\end{align} 
	see~\eqref{def_interval} for the definition of the intervals $\Gamma(\theta, \beta)$ and the associated boundary points $\gamma_\pm(\theta, \beta)$.
	This implies that $\Lambda_{\theta_k}\cap \Lambda_{\theta_{k+1}} \supsetneq \{(1,0)\}$, and hence, $\SO(2)\subsetneq \Tcal_\Ncal(R_\ast(\Omega))$ by the same arguments as above, as desired. 
	
	Finally, in order to verify~\eqref{intersection}, we observe first that
	\begin{align}\label{bound_comparison}
		\gamma_+(\theta_{k+1},\beta) \geq \gamma_-(\theta_k,\beta)
	\end{align} 
	for any $\beta\in [\max\{\sin\theta_k,\sin \theta_{k+1}\}, 1]$; 
	indeed, by the properties of the cotangent, the left-hand side is always non-negative and the right-hand side is non-positive, as $\theta_k \in (0,\frac{\pi}{2})$ and $\theta_{k+1} \in (\frac{\pi}{2},\pi)$ due to~\eqref{angles_perp}. 
	On the other hand, let  
	\begin{align*}
		d(\beta) := \gamma_+(\theta_k,\beta) - \gamma_-(\theta_{k+1},\beta) \text{ for }\beta\in [\max\{\sin\theta_k,\sin \theta_{k+1}\},1].
	\end{align*}
	From \eqref{trivial_interval}, one obtains immediately that $d(1)=0$, and $d'(1) <0$ follows from the calculation
	\begin{align*}
		d'(1) &= 
					\frac{\dd}{\dd \beta}\restrict{\beta=1} \big( \gamma_+(\theta_k,\beta) - \gamma_-(\theta_{k+1},\beta)\big)= -\cot \theta_k + \tan \theta_k + \cot \theta_{k+1} - \tan \theta_{k+1} \\
			&= \left(\cos \theta_{k+1}\cos\theta_k + \sin\theta_{k+1}\sin\theta_k\right)\left(\frac{1}{\sin\theta_{k+1}\cos\theta_k}-\frac{1}{\cos\theta_{k+1}\sin\theta_k}\right)\\
			&= -\frac{\sin(\theta_{k+1} -\theta_k)\cos(\theta_{k+1}-\theta_k)}{\sin\theta_{k+1}\sin\theta_k\cos\theta_{k+1}\cos\theta_k},
	\end{align*}
	under consideration of~\eqref{angles_perp}. One can therefore find a $\beta\in [\max\{\sin\theta_k,\sin \theta_{k+1}\},1)$ with $d(\beta)>0$, that is,
	\begin{align}\label{strict_bounds_inequality}
		\gamma_+(\theta_k,\beta)>\gamma_-(\theta_{k+1},\beta). 
	\end{align}
	Consequently, the combination of~\eqref{strict_bounds_inequality} with~\eqref{bound_comparison} gives that the intersection in~\eqref{intersection} is in fact trivial, which finishes the proof of the necessity.
\end{proof}

\begin{remark}
\label{rem:first_angle}
We continue with two brief comments on the assumption that $(\Omega,R_\ast)$ is a polycrystal as in Proposition~\ref{prop:taylorbound}. \medskip

	$a)$ Note that considering only orientations induced by rotations with angles in $[0, \pi)$ (see~\eqref{ordering}) is no real restriction, considering that plastic glide along the slip system is not uni-directed.
	Formally, we have that $\Ncal R_\theta^T = \Ncal R_{\theta\pm \pi}^T$ for any $\theta\in [0, \pi)$, and analogously, for $\Mcal$ in place of $\Ncal$. \medskip
	
	$b)$ The postulate in~\eqref{ordering} that the image $R_\ast(\Omega)$ contains the identity matrix can be made without loss of generality. 
	If $R_\ast(\Omega) = \{R_{\theta_1},\ldots,R_{\theta_N}\}$ for angles $\theta_1 <\ldots < \theta_N$ with $\theta_N-\theta_1<\pi$, but not necessarily $\theta_1=0$,	it is not hard to see that the results on the Taylor bounds in the case $\theta_1=0$ carry over to this more general setting.
	Indeed, with the new texture $\tilde{R}_\ast$ on $\Omega$ given by $\tilde{R}_\ast = R_\ast R_{\theta_1}^T$, we observe that $\tilde{R}_\ast(\Omega) = \{R_{\theta_1-\theta_1},\ldots,R_{\theta_N-\theta_1}\}$ and
	\begin{align*}
		\Tcal_\Ncal(\tilde{R}_\ast(\Omega)) = \bigcap_{x\in\Omega}\Ncal \tilde{R}_\ast(x)^T = \bigcap_{x\in\Omega} \Ncal R_{\theta_1}R^T_\ast(x) =\Bigl(\bigcap_{x\in\Omega} \Ncal R^T_\ast(x)\Bigr) R_{\theta_1} = \Tcal_\Ncal(R_\ast(\Omega)) R_{\theta_1};
	\end{align*}
	and the same for $\Mcal$ instead of $\Ncal$.
\end{remark}

To illustrate the findings of this section, we discuss implications for random polycrystals with uniformly distributed orientations. 
In particular, when the number of grains diverges, then it turns out that the Taylor bound for \eqref{relaxed} is trivial almost surely.
For an analysis of random polycrystals in the context of shape-memory alloys, see e.g., \cite{BaS05}.

\begin{example}[Randomized polycrystals]\label{ex:random_crystals}
	Let $(\Omega^k,R_\ast^k)$ for each $k\in\N$ be a polycrystal with at most $(k+1)$ grains. 
	In this example, the image $R_\ast^k(\Omega^k)$ is supposed to consist of the identity matrix $\Id$ and $k$ rotations $R_{\theta_1},\ldots,R_{\theta_{k}}$ with uniformly distributed angles $\theta_1, \ldots \theta_{k} \in(0,\pi)$;
	note that we do not impose any ordering of these angles. 
	
	The basic observation for our analysis is that $\Tcal_\Ncal(R_\ast^k(\Omega^k))=\SO(2)$ if $(\theta_1,\ldots,\theta_{k})\in (0,\pi)^k$ lies in
	\begin{align*}
		\Tcal_k:=\big \{(\theta_1,\ldots,\theta_{k}) \in (0,\pi)^k : \textstyle\bigcap_{i=1}^{k} \Ncal R_{\theta_i}^T \cap \Ncal = \SO(2)\big\}.
	\end{align*}
	With the help of Corollary \ref{cor:trivial_taylorbound} and the elementary calculations in Lemma \ref{lem:random_crystals_app}, we conclude that
	\begin{align*}
		\mu_k(\Tcal_k) = 1 - \frac{k+1}{2^k},
	\end{align*}
	where $\mu_k:=\frac{1}{\pi^k}\lambda_k$ and $\lambda_k$ is the $k$-dimensional Lebesgue measure.
	Hence, $\mu_k(\Tcal_k) \to 1$ as $k\to \infty$. This means that a myriad of grains with uniformly distributed slip directions renders the Taylor bound almost surely trivial.
\end{example}

We conclude this subsection with a comparison of the Taylor bounds and another geometry-independent inner approximation resulting from attainable affine boundary values of a canonically associated homogeneous problem, see \eqref{hom_cap} below. 
The basis of this discussion is the combination of our previous results on $\Tcal_\Mcal(R_\ast(\Omega))$ and $\Tcal_\Ncal(R_\ast(\Omega))$ with the well-established theory of homogeneous inclusions.

\begin{remark}[Alternative inner bound via a homogeneous inclusion]
	Instead of characterizing globally affine solutions to \eqref{unrelaxed} or \eqref{relaxed} as done to obtain the Taylor bounds, we now consider Lipschitz solutions to the homogeneous differential inclusion
	\begin{align*}
		\begin{cases}
			\displaystyle\nabla u \in \bigcap_{x\in\Omega} \Mcal R_\ast^T(x)=\Tcal_\Mcal(R_\ast(\Omega)) &\text{ a.e.~in } \Omega,\\
			u = Fx &\text{ on } \partial\Omega,
		\end{cases}\tag{$H_\Mcal$}\label{hom_cap}
	\end{align*}
	where $u\in W^{1,\infty}(\Omega;\R^2)$ is the unknown and $F\in\R^{2\times 2}$.
	Then, the set 
	\begin{align*}
		\Hcal(R_\ast(\Omega)) := \{F\in\R^{2\times 2} : \text{ there exists a solution $u\in W^{1,\infty}(\Omega;\R^2)$ to~\eqref{hom_cap}}\},
	\end{align*}
 	constitutes a geometry-independent inner bound for $\Fcal_\Mcal(\Omega,R_\ast)$, which, however, does not improve $\Tcal_\Ncal(R_\ast(\Omega))$ if $(\Omega, R_\ast)$ has more than one grain. 
	According to a classical result on differential inclusions as stated, e.g., in~\cite[Theorem 4.10]{Mue99} (applicable in view of the compactness of $\Tcal_\Mcal(R_\ast(\Omega))$ by Remark \ref{rem:taylor}\,$f)$) one has that 
	\begin{align*}
		\Hcal(R_\ast(\Omega)) \subset \Tcal_\Mcal(R_\ast(\Omega))^\qc \subset \Tcal_\Mcal(R_\ast(\Omega))^\pc \subset \Tcal_\Ncal(R_\ast(\Omega))^\pc=\Tcal_\Ncal(R_\ast(\Omega)),
	\end{align*}
	with the last identity due to Remark \ref{rem:taylor}\,$b)$.  
	The inclusion $\Hcal(R_\ast(\Omega)) \subset\Tcal_\Ncal(R_\ast(\Omega))$ is in general even strict.
	For example, if $R_\ast(\Omega) = \{\Id, R_{\frac{\pi}{6}},R_{\frac{5\pi}{6}}\}$, then $\Tcal_\Mcal(R_\ast(\Omega))$ is trivial by Remark \ref{rem:taylor}\,$f)$ so that $\Hcal(R_\ast(\Omega))=\SO(2)$, while $\Tcal_\Ncal(R_\ast(\Omega))\supsetneq \SO(2)$ by Corollary~\ref{cor:trivial_taylorbound}.
\end{remark}

\section{Outer bounds resulting from boundary grains}\label{sec:outer_bounds}
\subsection{Generalized Hadamard jump conditions}\label{sec:hadamard}
By the classical Hadamard jump condition, the gradients of any continuous and finitely piecewise affine function $u:\R^2\to \R^2$ need to be suitably rank-one connected.
More precisely, if the sets $\{\nabla u = A\}$ and $\{\nabla u = B\}$ for matrices $A,B\in\R^{2\times 2}$ are separated by a line with a normal $\nu\in\Scal^1$, then there exists some $a\in\R^2$ such that
\begin{align*}
	A-B = a\otimes \nu. 
\end{align*}
Following-up the seminal paper by Ball \& James~\cite{BaJ87}, this fundamental result has seen several further generalizations over the years.
For instance, in the case of Lipschitz maps, rank-one compatibility conditions between the polyconvex hulls of 
the sets of essential gradients (i.e., the smallest closed set containing all gradients up to a set of measure zero) on different sides of an interface are established in~\cite{BaC, IVV02}. 
The authors of~\cite{BaK14, BaM09} show under additional regularity assumptions, namely continuous differentiability up to the interfacial boundary or  locally bounded variation of the gradients, that the (approximate) gradients along the interface are rank-one connected pointwise (almost everywhere).
Recently, the Hadamard jump condition was investigated in the context of moving interfaces in~\cite{DeP19}; we also refer to this paper for a more detailed overview of the history of the problem.

The following result is a slight reformulation of \cite[Corollary 4]{IVV02} by Iwaniec, Verchota \& Vogel in the terminology of Definition~\ref{def:compatibility}, which has been suitably adapted to the needs of this paper using the translation- and rotation-invariance of the polyconvex hull of a set $\Bcal\subset\R^{2\times 2}$ in the sense that $(\Bcal R)^{\rm pc} =\Bcal^{\rm pc} R$ and $(\Bcal+A)^\pc=\Bcal^\pc + A$ for any $A\in \R^{2\times 2}$ and any $R\in \SO(2)$. 

\begin{theorem}[Generalized Hadamard jump condition for planar interfaces]\label{theo:hadamard_planar}
	Let $\Bcal \subset \R^{2\times 2}$ be closed, $A\in \R^{2\times 2}$, and $\nu\in\Scal^1$.
	If $u\in W^{1, \infty}(B(0,1);\R^2)$ satisfies 
	\begin{align*}
		\begin{cases}
			\nabla u\in \Bcal &\text{ a.e.~in } B^+_\nu(0,1),\\
			\nabla u = A &\text{ a.e.~in } B^-_\nu(0,1),
		\end{cases}		
	\end{align*}
	then $A$ is $\nu$-compatible with $\Bcal^\pc$.
\end{theorem}

The next theorem can be seen in turn as a special case of the work by Ball \& Carstensen \cite{BaC}, often cited in the literature, e.g., in~\cite{BaC99, BaC17, BaK14, BaM09}.  
However, to the best of our knowledge, the reference~\cite{BaC} has not yet been published, which is why we include here a detailed proof for the reader's convenience. 
The overall strategy \cite{Bal18} combines a blow-up argument with Theorem~\ref{theo:hadamard_planar}.

\begin{theorem}[Generalized Hadamard jump condition for curved interfaces]\label{theo:hadamard_curved}
	Let $U\subset \R^2$ be a bounded Lipschitz domain such that $\overline{U} = \overline U_1 \cup \overline U_2$ for two disjoint Lipschitz domains $U_1,U_2\subset U$ with interface $\Gamma :=\partial U_1 \cap \partial U_2 \cap U$.
	Further, let $\Bcal\subset \R^{2\times 2}$ be a closed set, $A\in\R^{2\times 2}$, and suppose that $x_0\in \Gamma$ is a point where the outer unit normal $\nu(x_0)$ of $U_1$ exists.
	If there is a function $u\in W^{1,\infty}(U;\R^2)$ with 
	\begin{align*}
		\begin{cases}
			\nabla u \in \Bcal &\text{ a.e.~in $U_2$,}\\
			\nabla u = A &\text{ a.e.~in $U_1$,}
		\end{cases}		
	\end{align*}
	then $A$ is $\nu(x_0)$-compatible with $\Bcal^\pc$.	
\end{theorem}

\begin{proof}
For $\eps>0$ sufficiently small, consider
\begin{align*}
v_\epsilon: B(0,1)\to \R^2, \quad v_\epsilon(x) = \frac{1}{\epsilon}\bigl(u(x_0+\epsilon x)-u(x_0)\bigr),
\end{align*}
observing that $\nabla v_\epsilon = \nabla u(x_0+\epsilon \,\cdot)$ and $v_\epsilon(0) = 0$. 
Since $u\in W^{1,\infty}(U;\R^2)$, the sequence $(v_\epsilon)_\epsilon$ is bounded in $W^{1,\infty}(B(0,1);\R^{2})$. Consequently, there exists a compact set $K\subset \R^{2\times 2}$ such that $\nabla v_\eps \in K$ a.e.~in $B(0,1)$, and one can extract a subsequence of $(v_\eps)_\eps$ (not relabeled) such that
\begin{align}\label{weakstar_lipschitz}
	v_\eps \weaklystar v \quad \text{ in } W^{1,\infty}(B(0,1);\R^2)
\end{align}
with $v\in W^{1,\infty}(B(0,1);\R^2)$. 
With the transformation $\psi_\eps: B(x_0,\eps)\to B(0,1),\: x\mapsto \frac{1}{\eps}(x- x_0)$, one has by assumption that $\nabla v_\eps\in \Bcal\cap K$ a.e.~in $\psi_\eps(U_1 \cap B(x_0, \eps))$ and $\nabla v_\eps=A$ a.e.~in $\psi_\eps(U_2\cap B(x_0, \eps))$.
Next, we prove that
\begin{align}\label{measure_conv1}
	\begin{split}
		\dist(\nabla v_\eps, \Bcal\cap K) \to 0 \quad &\text{ in measure on } B^+_{\nu(x_0)}(0,1),\\
		|\nabla v_\eps - A| \to 0 \quad &\text{ in measure on } B^-_{\nu(x_0)}(0,1)
	\end{split}
\end{align}
as $\eps\to 0$. With
$S_\eps:=(U_1\triangle B_{\nu(x_0)}^-(x_0, \eps)) \cap B(x_0, \eps)= (U_2 \triangle B_{\nu(x_0)}^+(x_0, \eps)) \cap B(x_0, \eps)$, where $\triangle$ denotes the symmetric difference between two sets (see Figure \ref{fig:graph}), the convergences in~\eqref{measure_conv1} follow immediately from
\begin{align}\label{psieps}
	|\psi_\eps(S_\eps)|\to 0 \quad \text{as $\eps\to 0$.}
\end{align} 
To see the latter, recall that the interface $\Gamma$ is locally the graph of a Lipschitz function and the unit outer normal $\nu(x_0)$ of $U_1$ exists, so that
\begin{align*}
	\Gamma\cap B(x_0,\eps) = \{x_0  - t \nu(x_0)^\perp + g(t)\nu(x_0) : t\in (-\eps,\eps)\}\cap B(x_0,\eps)
\end{align*} 
with $g: (-\eps, \eps) \to \R$ Lipschitz continuous satisfying $g(0)=0$ and $g'(0) = 0$. Then, 
\begin{align*}
	\big|\psi_\eps(S_\eps) \bigr|=( \det \nabla \psi_\eps)\, |S_\eps| \leq \eps^{-2} \int_{-\eps}^\eps |g(t)| \dd t \leq \sup_{t\in (-\eps,\eps)} \left|\frac{g(t)}{t}\right|
 \to 0 \quad\text{ as } \eps \to 0, 
\end{align*}
implying~\eqref{psieps}. 

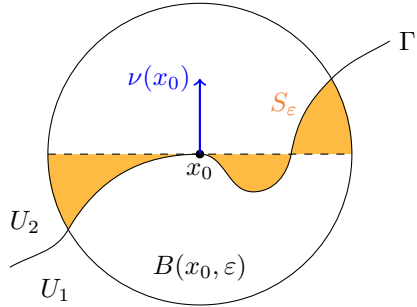
\begin{figure}[h!]
	\centering
	\begin{tikzpicture}
		\coordinate (a) at ({-2*cos(30)},{-2*sin(30)});
		\coordinate (b) at ({2*cos(30)},{2*sin(30)});
		\fill [Dandelion](a) to [out=60,in=180] (0,0) to [out=0,in=180] (0.7,-0.5) to [out=0,in=260] (1.2,0) to [out=180,in=0] (-2,0) to [out=-90,in=120] (a);
		\fill [Dandelion] (1.2,0) to [out=80,in=225] (b) to [out=-60,in=90] (2,0) to cycle;	
		\draw (0,0) circle (2);
		\draw (0,-1.8) node [anchor=south] {$B(x_0,\eps)$};
		\draw [dashed] (-2,0) -- (2,0);
		\draw [fill=black] (0,0) circle (0.05);
		\draw (0,0) node [anchor=north] {$x_0$};
		\draw [->,blue,thick] (0,0) -- (0,1);
		\draw (0,1) node [blue,anchor=east] {$\nu(x_0)$};
		\draw (-2.5,-1.5) to [out=30,in=240] (a) to [out=60,in=180] (0,0) to [out=0,in=180] (0.7,-0.5) to [out=0,in=260] (1.2,0) to [out=80,in=225] (b) to [out=45,in=210] (2.5,1.5);
		\draw (0.8,0.4) node [Orange,anchor=south west] {$S_\eps$};
		\draw (-2.3,-0.9) node {$U_2$};
		\draw (-1.9,-1.8) node {$U_1$};
		\draw (2.5,1.5) node [anchor=west] {$\Gamma$};
	\end{tikzpicture}
	\caption{An illustration of the interface $\Gamma$ in a neighborhood of $x_0$, where the orange area depicts the set $S_\eps$.}\label{fig:graph}
\end{figure}

The remaining proof uses classical arguments from Young measure theory, see, e.g.,~\cite{FoL07, Mue99, Rin18} for general introductions. 
If  $\{\mu_x\}_{x\in B(0,1)}$ is the gradient Young measure generated by a (non-relabeled) subsequence of $(\nabla v_\eps)_\eps$, we infer along with~\eqref{measure_conv1}, 
 \begin{align}\label{support_GYM}
	\begin{split}
		\supp \mu_x \subset \Bcal \cap K &\quad \text{ for a.e.~} x\in B^+_{\nu(x_0)}(0,1),\\
		\supp \mu_x = \{A\}&\quad\text{ for a.e.~} x\in B^-_{\nu(x_0)}(0,1),
	\end{split}
\end{align}
and it holds for the barycenters of $\mu_x$ that $\int_{\R^{2\times 2}} M \dd \mu_x(M)=\nabla v(x)$ for a.e.~$x\in B(0,1)$ in view of \eqref{weakstar_lipschitz}. 
Since the quasiconvex hull $(\Bcal\cap K)^\qc$ consists of the barycenters of homogeneous gradient Young measures with support in $\Bcal\cap K$, and $\mu_x$ is a homogeneous gradient Young measure itself  for almost every $x$, it follows from~\eqref{support_GYM} that
\begin{align*}
\begin{cases}
	\nabla v \in (\Bcal\cap K)^\qc & \text{ a.e.~in } B^+_{\nu(x_0)}(0,1),\\
	\nabla v= A & \text{ a.e.~in } B^-_{\nu(x_0)}(0,1).
	\end{cases}
\end{align*}

Finally, the statement follows from Theorem \ref{theo:hadamard_planar}, according to which $A$ is $\nu(x_0)$-compatible with $((\Bcal\cap K)^\qc)^\pc\subset \Bcal^\pc$, noting in particular that $(\Bcal\cap K)^\qc$ is compact as the quasiconvex hull of a compact set. 
\end{proof}

\subsection{Application to polycrystals}\label{sec:application_outer}
We now apply the results from Section \ref{sec:hadamard} on rank-one compatibility along curved interfaces to the boundary grains of the polycrystal $(\Omega,R_\ast)$. 
Therefore, in light of~\eqref{Sl2_bound}, let us introduce the set
\begin{align}\label{TdNcal}
	\Tcal^\partial_\Ncal(\Omega,R_\ast) := \{F\in \Sl(2): F \text{ is $\nu(x)$-compatible with $\Ncal R^T_\ast(x)$ for every $x\in \bdpoly$}\},
\end{align}
with $\nu$ the outer unit normal of $\Omega$ and $\bdpoly$ as in~\eqref{bdpoly_def}; the texture $R_\ast$ is canonically extended to $\partial\Omega$ except in the boundary dual points, cf.~\eqref{bddual_def}.
Applying Theorem \ref{theo:hadamard_curved} locally around every $x\in \bdpoly$ yields the outer bound
\begin{align*}
	\Fcal_\Mcal(\Omega,R_\ast)\subset \Tcal^\partial_\Ncal(\Omega,R_\ast),
\end{align*}
considering that $(\Mcal R_\ast^T(x))^\pc = (\Mcal_{R_\ast(x) e_1})^\pc = \Ncal_{R_\ast(x) e_1} = \Ncal R^T_\ast(x)$.
\medskip

Notice that requiring the rank-one connectedness at boundary dual points does not improve the outer bound \eqref{TdNcal}. Suppose that $x\in \partial \Omega$ is a boundary dual point where exactly two neighboring grains, say $\Omega_k$ and $\Omega_l$, meet.  
If $\nu(x)$ exists, then Theorem \ref{theo:hadamard_curved} along with Proposition \ref{prop:qc_hull} and~\eqref{twograins} shows that any element of $\Fcal_\Mcal(\Omega, R_\ast)$ is $\nu(x)$-compatible with
\begin{align*}
	\Big(\Mcal R_\ast^T\restrict{\Omega_k}\cup \Mcal R_\ast^T\restrict{\Omega_l}\Big)^\pc = \Sl(2),
\end{align*}
which is a trivial statement since $\Sl(2)$ is already an outer bound for all polycrystals. 
In case that more than two grains meet in $x$, the argument is analogous. 
\medskip

A (possibly) larger outer bound on the set of attainable macroscopic strains of the polycrystal is $\Tcal^\perp_\Ncal(\Omega,R_\ast)$, which we define as in~\eqref{TdNcal}, but with 
\begin{align*}
\bdperp=\{x\in \bdpoly : \nu(x)\cdot R_\ast(x)e_1=0 \}\end{align*} 
in place of $\bdpoly$. 
It is easier to characterize than $\Tcal^\partial_\Ncal(\Omega,R_\ast)$ in practice since it accounts for rank-one compatibility only at specific boundary points where an orthogonality condition between the outer unit normal and the slip orientations is satisfied.
Even though the inclusion
\begin{align*}
	\Tcal^\partial_\Ncal(\Omega,R_\ast)\subset\Tcal^\perp_\Ncal(\Omega,R_\ast)
\end{align*}
is in general strict, as shown in Example \ref{ex:bicrystals}\,$a)$, it is possible to provide a sufficient geometric condition on the polycrystal that ensures the equality.
The next statement, a refinement of Proposition \ref{prop:taylor_bdr_intro} and a direct consequence of Remark~\ref{rem:rank_one}, gives more insight into the relation between these outer bounds and $\Ncal R_\ast^T$. 

\begin{proposition}\label{prop:taylor_bdr}
	Let $\Omega_1,\ldots, \Omega_M$ for $M\in\N$ be the boundary grains of the polycrystal $(\Omega,R_\ast)$ and let $J =\{i\in\{1,\ldots,M\}: \bdperp\cap \partial\Omega_i \neq \emptyset\}$.
	\smallskip
	
	$a)$ If $J=\{1,\ldots,M\}$, then it holds that $\Tcal^\partial_\Ncal(\Omega,R_\ast) = \Tcal^\perp_\Ncal(\Omega,R_\ast)$. 
	\smallskip
	
	$b)$  If $J\neq\emptyset$, then
	 \begin{align*}
		\Tcal^\perp_\Ncal(\Omega, R_\ast) = \bigcap_{i\in J} \Ncal R^T_\ast\restrict{\Omega_i}.
	\end{align*}	 

	$c)$ If $J\cup J'\neq \emptyset$ with $J'=\{i\in \{1, \ldots, M\}: \overline{\{\pm \nu(x) : x\in \bdpoly\cap\partial\Omega_i \}} =\Scal^1\}$, then
	\begin{align}\label{taylor_bdr2}
 		\Tcal^\partial_\Ncal(\Omega,R_\ast) \subset \bigcap_{i\in J\cup J'} \Ncal R^T_\ast\restrict{\Omega_i}.
	\end{align}
\end{proposition}
While $J$ selects the boundary crystals $\Omega_i$ whose outer unit normal $\nu$ to $\Omega$ is orthogonal to the associated slip direction at some point, the index set $J'$ identifies strongly curved boundary grains, where the image of $\pm \nu$ is dense in $\Scal^1$. 
In Example \ref{ex:bicrystals}, the combination of both these index sets enables a full characterization of $\Fcal_\Mcal(\Omega,R_\ast)$.
\medskip

\begin{remark}[Higher regularity of gradients] 
	Let us comment on how requiring higher regularity for the solutions to~\eqref{unrelaxed} affects the rank-one compatibility conditions at the boundary grains, and thus, the corresponding outer bounds for the effective strains.
	The answer depends on the geometry of the boundary grains of the polycrystal. 

	To be precise, let $\Omega$ be a $C^1$-domain and suppose that for $F\in\Fcal_\Mcal(\Omega,R_\ast)$ there is a solution $u\in C^1(\overline{\Omega};\R^2)$ to \eqref{unrelaxed}.
	It then follows from \cite[Theorem 3.2]{BaK14} that $F$ lies in
	\begin{align*}
		\Tcal^\partial_\Mcal(\Omega,R_\ast) :=\{F\in\Sl(2): \text{$F$ is $\nu(x)$-compatible with $\Mcal R_\ast^T(x)$ for every $x\in\bdpoly$}\}. 
	\end{align*}
	
 	If $\bdperp=\emptyset$, one observes that $\Tcal^\partial_\Mcal(\Omega,R_\ast)= \Tcal^\partial_\Ncal(\Omega,R_\ast)$ due to $(i)\Leftrightarrow(ii)$ in Lemma \ref{lem:rank-one}, meaning that the added regularity does not sharpen the outer bound in this case. 
 
	On the other hand, if there exists $x\in\bdperp$, Remark \ref{rem:rank_one} implies that $F\in \Mcal R_\ast^T(x)$, instead of merely $F\in \Ncal R_\ast^T(x)$.
	A comparison of Remark \ref{rem:taylor}\,$f)$ and Corollary~\ref{cor:trivial_taylorbound} therefore underlines that the polycrystal behaves more rigidly under the assumption of higher regularity. 
\end{remark}

Another natural outer bound for $\Fcal_\Mcal(\Omega,R_\ast)$ can be derived dropping the gradient structure in the differential inclusion in~\eqref{unrelaxed}; for an analogous approach in the context of polycrystalline shape-memory alloys, see also~\cite[page 125]{BaK97}. 
Yet, it turns out that this bound is trivial due to the unboundedness of $\Mcal$, and hence, unfit to improve $\Tcal^\partial_\Ncal(\Omega,R_\ast)$ and $\Tcal^\perp_\Ncal(\Omega,R_\ast)$.

\begin{remark}\label{rem:curl-free}
	Let us consider the inclusion problem
	\begin{align*}
		\begin{cases}
			U(x) \in \Mcal R_\ast^T(x) &\text{ for a.e.~} x\in \Omega,\\
			\int_\Omega U(x) \dd x = F|\Omega|,
		\end{cases}\tag{$C_\Mcal$}\label{no_gradient}
	\end{align*}
	with the unknown $U\in L^\infty(\Omega;\R^{2\times 2})$ and $F\in\R^{2\times 2}$, which arises from \eqref{unrelaxed} in disregard of the gradient structure. 
	We observe that
	\begin{align*}
		\Ccal_\Mcal(\Omega,R_\ast) = \big\{F\in \R^{2\times 2} : \text{ there exists a solution $U\in L^\infty(\Omega;\R^2)$ to }\eqref{no_gradient} \big\} 
	\end{align*}
constitutes an outer bound of $\Fcal_\Mcal(\Omega,R_\ast)$, given that \eqref{no_gradient} clearly admits a larger class of solutions than \eqref{unrelaxed}. 
	However, under the assumption that $(\Omega,R_\ast)$ is not a single crystal, it holds that
	\begin{align*}
		\Ccal_\Mcal(\Omega,R_\ast) =  \R^{2\times 2}.
	\end{align*}
	This follows from
		\begin{align}\label{bound_Fcal}
		\R^{2\times 2} = \sum_{i=1}^N \frac{|\Omega_i|}{|\Omega|}\Mcal_{R_\ast\restrict{\Omega_i}e_1}^{\rm co}  = \sum_{i=1}^N \frac{|\Omega_i|}{|\Omega|}(\Mcal R_\ast^T\restrict{\Omega_i})^{\rm co} \subset \Ccal_\Mcal(\Omega,R_\ast),
	\end{align}
	where $(\cdot)^{\rm co}$ denotes the convex hull and $\Omega_1,\ldots,\Omega_N\subset\Omega$ with $N\geq 2$ are the grains of $\Omega$. 
	The first identity in \eqref{bound_Fcal} relies on the simple fact that $\Mcal_s^{\rm co} = \{F\in\R^{2\times 2} : |Fs|\leq 1\}$ for any $s\in\Scal^1$, while the last inclusion is based on standard convexity arguments; essentially, it suffices to employ Carath\'eodory's theorem in combination with a suitable refinement of the partition given by the grains.
\end{remark}

In the following, we present examples of polycrystals for which the set of attainable macroscopic strains $\Fcal_\Mcal(\Omega,R_\ast)$ can be fully characterized with the help of the previous concepts. The analysis of polycrystalline structures with a symmetric material response and their rigidity is followed up by a brief discussion of selected bicrystals.

\begin{example}[Polycrystals with sufficient symmetry]\label{ex:symmetric_outer}
	In analogy to \cite{BaK97, KoN00}, we say that $(\Omega,R_\ast)$ has sufficient symmetry if there exists $R\in\SO(2)\setminus\{\pm \Id\}$ such that $F\in\Fcal_\Mcal(\Omega,R_\ast)$ if and only if $R^TFR \in \Fcal_\Mcal(\Omega,R_\ast)$,	or equivalently, due to $R\Mcal=\Mcal$,
	\begin{align}\label{suff_symm2}
		\Fcal_\Mcal(\Omega, R_\ast) = \Fcal_\Mcal(\Omega,R_\ast)R^T.
	\end{align}		
	\smallskip
	
	$a)$ Let $\Omega_1,\ldots,\Omega_M$ for $M\in\N$ be the boundary grains of $(\Omega,R_\ast)$ and let $J,J'\subset \{1,\ldots,M\}$ be as in Proposition~\ref{prop:taylor_bdr}. 
	If there is $i\in J\cup J'$, then \eqref{taylor_bdr2} in combination with \eqref{suff_symm2} and Lemma \ref{lem:symmetric} yields that
	\begin{align*}
		\Fcal_\Mcal(\Omega,R_\ast) \subset \bigcap_{k=0}^\infty \Ncal R_\ast^T\restrict{\Omega_i} (R^T)^k =  \Bigl(\bigcap_{k=0}^\infty \Ncal (R^k)^T\Bigr) R_\ast^T\restrict{\Omega_i}=  \SO(2).
	\end{align*}
	Since $\SO(2)$ is a trivial inner bound, the polycrystal is fully rigid in the sense that
	\eqref{unrelaxed} can only be solved with affine boundary values in $\SO(2)$. 

	In view of the geometry-independence of the Taylor bound, the arguments above also show that $\Tcal_\Ncal(R_\ast(\Omega))=\SO(2)$ for any polycrystal $(\Omega,R_\ast)$ with sufficient symmetry, even if $J\cup J'=\emptyset$.
	\medskip
	
	 $b)$ A special class of polycrystals with sufficient symmetry is given as follows. Under the assumption that $\Omega$  has center of in the origin, that is, $\int_\Omega x \dd x =0$, suppose that there exists a rotation $R\in\SO(2)\setminus\{\pm \Id\}$ such that
	\begin{align}\label{symmetric_crystal}
		R\Omega = \Omega\quad\text{and}\quad R_\ast(x)=R^TR_\ast(Rx)\quad  \text{for a.e.~}x\in \Omega,
	\end{align}
	where $R\Omega=\{Rx: x\in \Omega\}$. A simple transformation argument shows that~\eqref{suff_symm2} is indeed satisfied. 
	Figure \ref{fig:symmetric_crystal} depicts an easy example of a polycrystal that fulfills \eqref{symmetric_crystal} with $R=R_\frac{\pi}{2}$; 
	here, $J=\{1,2,3,4\}$, which implies full rigidity.
	\begin{figure}
		\centering
		\begin{tikzpicture}[scale=0.85]
			\fill [Cerulean] (0,0) --++ (225:2) arc (225:135:2);
			\fill [Cerulean] (0,0) --++ (-45:2) arc (-45:45:2);
			\fill [Dandelion] (0,0) --++ (45:2) arc (45:135:2);
			\fill [Dandelion] (0,0) --++ (-45:2) arc (-45:-135:2);
			\draw (0,0) --++ (45:2); \draw (0,0) --++ (135:2); 
			\draw (0,0) --++ (-45:2); \draw (0,0) --++ (-135:2); 
			\draw (0,0) circle (2cm);
			
			\draw [<-](-0.3,1.2) --++ (0:0.6);
			\draw [<-](-1.2,-0.3) --++ (90:0.6);
			\draw [->](-0.3,-1.2) --++ (0:0.6);
			\draw [->](1.2,-0.3) --++ (90:0.6);
			
			\draw (-1,2.1) node {$\Omega_1$};
			\draw (-2.1,-1) node {$\Omega_2$};
			\draw (1,-2.1) node {$\Omega_3$};
			\draw (2.1,1) node {$\Omega_4$};
		\end{tikzpicture}
		\caption{A visualization of the polycrystal $(\Omega,R_\ast)$ with $\Omega=B(0,1)$ and $R_\ast= \Id \ONE_{\Omega_1\cup\Omega_3} + R_\frac{\pi}{2}\ONE_{\Omega_2\cup\Omega_4}$.}\label{fig:symmetric_crystal}
	\end{figure}
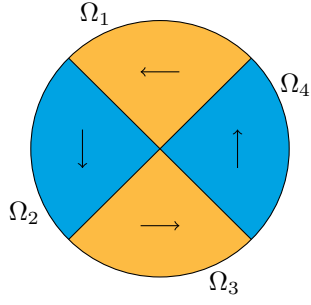
\end{example}

\begin{example}[Bicrystals]\label{ex:bicrystals}
	Let $\Omega=B(0,1)$ and fix slip directions $s,s'\in\Scal^1$ with $s\neq \pm s'$. \smallskip
	
	$a)$ Consider $(\Omega,R_\ast)$ with
	\begin{align*}
		 R_\ast e_1= \begin{cases}
			s  &\text{ in } \Omega_1= B^+_{e_2}(0,1),\\
			s' &\text{ in } \Omega_2= B^-_{e_2}(0,1),
		\end{cases}
	\end{align*}
	see Figure \ref{fig:bicrystal}\,$a)$. Recalling the definitions \eqref{taylor_sets} and \eqref{TdNcal}, Proposition \ref{prop:affine_lipschitz} and Proposition \ref{prop:taylor_bdr}\,$c)$ imply that
	\begin{align*}
		\Ncal_s\cap \Ncal_{s'}=\Tcal_\Ncal(R_\ast(\Omega)) \subset \Fcal_\Mcal(\Omega,R_\ast) \subset\Tcal^\partial_\Ncal(\Omega,R_\ast) = \Ncal_s \cap \Ncal_{s'},
	\end{align*}
	which determines the attainable macroscopic strains of the polycrystal. 
	If $s=e_2$, then $\bdperp$ consists of only a single point in $\partial \Omega_2\cap\bdpoly$ and 
	\begin{align*}
		\Tcal^\perp_\Ncal(\Omega,R_\ast) =\Ncal_{s'} \supsetneq	\Ncal_s\cap\Ncal_{s'} =\Tcal^\partial_\Ncal(\Omega,R_\ast).
	\end{align*}
	
	\begin{figure}[h!]
		\begin{subfigure}{0.49\linewidth}
			\centering
			\begin{tikzpicture}
				\fill [Cerulean] (-1.5,0) to [out=90,in=180] (0,1.5) to [out=0,in=90] (1.5,0) to cycle;
				\fill [Dandelion] (-1.5,0) to [out=-90,in=180] (0,-1.5) to [out=0,in=-90] (1.5,0) to cycle;
				\draw (0,0) circle (1.5);
				\draw (-1.5,0) -- (1.5,0);
				\draw (0,-0.75) --++ (210:0.3);
				\draw [->] (0,-0.75) --++ (30:0.3);
				\draw [->] (0,0.45) --++ (90:0.6);
				\draw (-1.3,1.3) node {$\Omega_1$}; 
				\draw (-1.3,-1.3) node {$\Omega_2$}; 
			\end{tikzpicture}
		\end{subfigure}
		\begin{subfigure}{0.49\linewidth}
			\centering
			\begin{tikzpicture}
				\begin{scope}[rotate=0]
				\coordinate (a) at ({-1.5*cos(20)},{-1.5*sin(20)});
				\coordinate (b) at ({1.5*cos(20)},{-1.5*sin(20)});
				\fill [Cerulean] (-1.5,0) to [out=90,in=180] (0,1.5) to [out=0,in=90] (1.5,0) to [out=-90,in=70] (b) to [out=180,in=0] (a) to [out=110,in=-90] (-1.5,0);
				\fill [Dandelion] (a) to [out=0,in=180] (b) to [out=250,in=0] (0,-1.5) to [out=180,in=-70] (a);
				\draw (0,0) circle (1.5);
				\draw (a) -- (b);
				\draw (0,-1) --++ (210:0.3);
				\draw [->] (0,-1) --++ (30:0.3);
				\draw [->] (0,0) --++ (90:0.6);
				\draw (-1.3,1.3) node {$\Omega_1$};
				\draw (-1.3,-1.3) node {$\Omega_2$};
				\end{scope}
			\end{tikzpicture}
		\end{subfigure}
		\caption{Illustration of a polycrystal as in $a)$ with $s=e_2$ and $s'=R_\frac{\pi}{6}e_1$, and $b)$ with $s=e_2$, $s'=R_\frac{\pi}{6}e_1$ and $\theta=\frac{\pi}{9}$.}\label{fig:bicrystal}
	\end{figure}
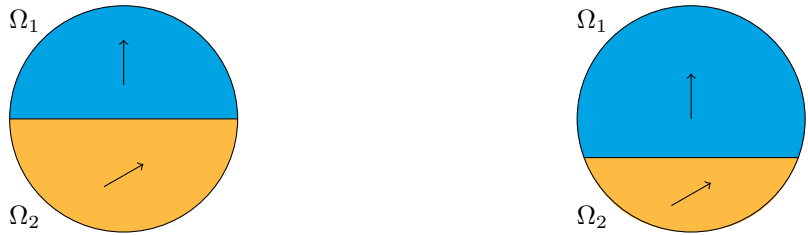

	$b)$ As a second example, let $(\Omega, R_\ast)$ be given by
	\begin{align*}
		R_\ast e_1=
		\begin{cases}
			s  &\text{ in } \Omega_1= \{x\in B(0,1): x_2> -\sin \theta\},\\
			s' &\text{ in } \Omega_2= \{x\in B(0,1): x_2< -\sin \theta\},
		\end{cases}
	\end{align*}
	for  $\theta\in (0,\frac{\pi}{2})$, see Figure \ref{fig:bicrystal}\,$b)$. 
	In this case, Proposition~\ref{prop:taylor_bdr} gives rise to the characterization of $\Fcal_\Mcal(\Omega,R_\ast)$ for 
	 $s'=R_\ffi e_1$ with $|\ffi|<\frac{\pi}{2}-\theta$, namely,
	\begin{align*}
		\Fcal_\Mcal(\Omega,R_\ast) = \Ncal_s \cap \Ncal_{s'}. 
	\end{align*}
\end{example}

\section{On the non-optimality of the Taylor bound}\label{sec:constructions}
We discuss in Example \ref{ex:bicrystals} several instances of polycrystals whose Taylor bound is optimal.
Moreover, Example \ref{ex:symmetric_outer} demonstrates that polycrystals with sufficient symmetry are fully rigid under suitable assumptions.
These results naturally raise the question of whether the Taylor bound is generally optimal. The next proposition provides a negative answer to this issue via the construction of a specific polycrystal.
Our geometric setup is mainly inspired by the rotated-square approach discussed in \cite{CKZ17, Pom10} in the context of stress-free martensitic inclusions in the theory of shape-memory alloys. 

\begin{proposition}\label{prop:optimality}
	There exists a polycrystal $(\Omega,R_\ast)$ such that $\Tcal_\Ncal(R_\ast(\Omega)) \subsetneq \Fcal_\Mcal(\Omega,R_\ast)$.
\end{proposition}
\begin{proof}
	Let us start with a brief overview of our strategy for finding an explicit example of a polycrystal with non-optimal Taylor bound. 
	To keep the arguments simple, we aim for a polycrystal $(\Omega,R_\ast)$ with the two orthogonal slip directions $e_1$ and $e_2$, which guarantees a trivial Taylor bound according to Corollary \ref{cor:trivial_taylorbound}, i.e., $\Tcal_\Ncal(R_\ast(\Omega))=\SO(2)$.
	Necessarily, our task is then to determine an $F\in \Fcal_\Mcal(\Omega,R_\ast)\setminus \SO(2)$.
	As a first step, we construct a finitely piecewise affine solution to the homogeneous partial differential inclusion
	\begin{align}\label{hom_relaxed}
		\begin{cases}
			\nabla v \in \Ncal_{e_1} \cup \Ncal_{e_2} &\text{ a.e.~in } \Omega,\\
			v = Fx &\text{ on } \partial\Omega
		\end{cases}
	\end{align}
	for a suitable $F\notin \SO(2)$.	
	The non-empty connected components of the sets $\{\nabla v \in \Ncal_{e_1}\}$ and $\{\nabla v \in \Ncal_{e_2}\}$ are chosen as the polycrystalline grains $(\Omega_i)_i$, whose orientations are set accordingly to be $e_1$ and $e_2$; 
	on $\{\nabla v \in \SO(2)\}$, one may select either of these two orientations.
	This procedure provides a finitely piecewise affine solution to \eqref{relaxed}. 
	In the final step, we apply Proposition \ref{prop:affine_lipschitz} to obtain a Lipschitz solution to \eqref{unrelaxed} with the same affine boundary condition.
	\medskip
	
	Now, let $\Omega\subset \R^2$ be the square with corners $(0,0), (1,3), (4,2), (3,-1)$. 
	Observe that the specific choice of $\Omega$ yields that the outer unit normal $\partial\Omega$ never attains the values $e_1$ or $e_2$ so that $\Tcal^\perp_\Ncal(\Omega,R_\ast)=\Sl(2)$.
	Similarly to \cite{CKZ17, Pom10}, we subdivide $\Omega$ into eight triangles $T_1,...,T_8$ and a square $S$, which represent the pieces where the solution $v$ to the homogeneous problem \eqref{hom_relaxed} will be affine. 
	Precisely,
	\begin{align}\label{triangles_square}
		T_1 &= \{(0,0),(2,0),(1,1)\}^{\rm co},\quad &&T_2= \{(0,0),(1,1),(1,3)\}^{\rm co},\nonumber\\
		T_3 &= \{(1,1),(1,3),(2,2)\}^{\rm co},\quad &&T_4= \{(1,3),(2,2),(4,2)\}^{\rm co},\nonumber\\
		T_5 &= \{(2,2),(4,2),(3,1)\}^{\rm co},\quad &&T_6= \{(4,2),(3,1),(3,-1)\}^{\rm co},\\
		T_7 &= \{(2,0),(3,1),(3,-1)\}^{\rm co},\quad &&T_8= \{(0,0),(2,0),(3,-1)\}^{\rm co},\nonumber\\
		S&=\{(1,1),(2,2),(3,1),(2,0)\}^{\rm co},\nonumber
	\end{align}
	where $(\cdot)^{\rm co}$ denotes the classical convex hull, see Figure \ref{fig:layout}.
	We stress that the area of these subsets satisfy $|T_1|=\ldots = |T_8|=\frac{1}{2}|S|$ and do not represent the grains of the polycrystal.
	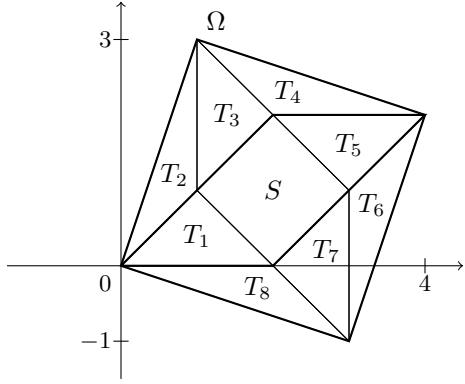
\begin{figure}[h!]
		\centering
		\begin{tikzpicture}
			\draw [->](-1.5,0)--(4.5,0);
			\draw [->](0,-1.5)--(0,3.5);
			\draw (-0.1,-1)--(0.1,-1);
			\draw (0,-1) node [anchor=east] {\small $-1$};
			\draw (-0.1,3)--(0.1,3);
			\draw (0,3) node [anchor=east] {\small $3$};
			\draw (4,-0.1) -- (4,0.1);
			\draw (4,0) node [anchor=north] {\small $4$};
			\draw (0,0) node [anchor=north east] {\small $0$};

			\draw [smooth, line join=round](1,1)--++(1,1)--++(1,-1)--++(-1,-1)--cycle; //S
			\draw (2,1) node {$S$};
	
			\draw (1,1) --++ (1,1)--++(-1,1)--cycle; //T1
			\draw (2,0)--++(1,-1)--++(0,2)--cycle; //T5
			\draw (1,0.4) node {$T_1$};
			\draw (3,1.6) node {$T_5$};
			
			\draw [smooth, line join=round](0,0) --(1,1)--++(0,2)--(0,0); //T2
			\draw [smooth, line join=round](3,-1)--++(1,3)--++(-1,-1)--cycle; //T6
			\draw (0.7,1.2) node {$T_2$};
			\draw (3.3,0.8) node {$T_6$};
			
			\draw [smooth, line join=round](0,0)--(2,0)--(1,1)--(0,0); //T3
			\draw [smooth, line join=round](2,2)--++(2,0)--++(-1,-1)--cycle; //T7
			\draw (1.4,2) node {$T_3$};
			\draw (2.7,0.2) node {$T_7$};
			
			\draw [smooth, line join=round](2,2)--++(2,0)--++(-3,1)--cycle; //T4
			\draw [smooth, line join=round](2,0)--++(1,-1)--(0,0)--(2,0); //T8
			\draw (2.2,2.3) node {$T_4$};
			\draw (1.8,-0.3) node {$T_8$};
			
			\draw (1,3) node [anchor = south west] {$\Omega$};
			\draw [thick](0,0) --++(1,3)--++(3,-1)--++(-1,-3)--cycle;
			\draw [thick](0,0) --++(2,2)--++(2,0)--++(-2,-2)--++(-2,0);
		\end{tikzpicture}
		\caption{The domain $\Omega$ and its partition into the eight triangles $T_1,\ldots,T_8$ and the square $S$ in the center.}\label{fig:layout}
	\end{figure}
	
	Starting the construction of the solution $v$ to \eqref{hom_relaxed}, we set 
	\begin{align}\label{S}
		\nabla v\restrict{S} = \begin{pmatrix} 1 & \gamma \\ 0 & 1\end{pmatrix} \in \Mcal_{e_1}\subset \Ncal_{e_1},
	\end{align}
	which means that $S$ undergoes a shear in $e_1$-direction with shear parameter $\gamma\in \R$ to be specified later. 

	Next, we explain how to construct the gradients in $T_1$ and $T_5$.  
	Inspired by the point-symmetry of the partition \eqref{triangles_square} of $\Omega$, we apply affine deformations with identical strain on both these sets.
	To obtain a continuous deformation, the gradients of $v$ need to be rank-one compatible along the interfaces of $S$ and the other triangles;
	in particular,  
	\begin{align}\label{33}
		\nabla v\restrict{T_1} (e_1 - e_2) = \nabla v\restrict{S} (e_1-e_2) = \nabla v\restrict{T_5}(e_1 - e_2)=\begin{pmatrix}1-\gamma\\-1\end{pmatrix},
	\end{align}
	which prescribes the affine deformation $v$ on $T_1$ and $T_5$ along $e_1-e_2$ up to translation.
	To fully pin down the construction, we need to designate $v$ along another linearly independent direction. 
	A natural choice for the latter are $e_1$ and $e_1+e_2$  since they are parallel to the other edges of $T_1$ and $T_5$.
	Considering that $\Ncal_{e_1}\subset \Sl(2)$ necessitates that the restriction of $v$ to these two sets needs to be incompressible.	
	Our approach to finding such a locally volume-preserving deformation works via extension of the images of the edges of $S$ under $v$ in a way that gives rise to 
	\begin{align*}
		 \nabla v \restrict{T_5}e_1 = \nabla v\restrict{T_1}e_1 = \frac{1}{2}\nabla v\restrict{S}(e_1+e_2)= \frac{1}{2}\begin{pmatrix} 1+\gamma\\ 1\end{pmatrix}. 
	\end{align*}
	Hence, together with~\eqref{33},
	\begin{align}\label{T15}
		\nabla v \restrict{T_5} = \nabla v \restrict{T_1} = \frac{1}{2}\begin{pmatrix}1+\gamma & 3\gamma-1 \\1 & 3 \end{pmatrix}.
	\end{align}
	If $\gamma\in [-1-\sqrt{3}, \sqrt{3}-1]$, then the right-hand side is contained in $\Ncal_{e_1}$.
	The same strategy applied to $T_3$ and $T_7$ yields that
	\begin{align}\label{T37}
		\nabla v \restrict{T_3} = \nabla v \restrict{T_7} = \frac{1}{2}\begin{pmatrix}3+\gamma & \gamma-1\\1 & 1\end{pmatrix} \in \Ncal_{e_2},
	\end{align}
	if $\gamma\in[1-\sqrt{3},1+\sqrt{3}]$; hence, we take $\gamma\in [-\sqrt{3}+1, \sqrt{3}-1]$ from now on. 
	
	Having $v$ fixed on the triangles $T_1,T_5, T_3,$ and $T_7$, the sought finitely piecewise affine solution to~\eqref{hom_relaxed} is automatically determined on all of $\Omega$. 
	Indeed, the rank-one compatibility along interfaces combined with the constructions \eqref{T15} and \eqref{T37} require that
	\begin{align}\label{T2468}
		\nabla v \restrict{T_6} = \nabla v \restrict{T_2} = \frac{1}{2}\begin{pmatrix}1+3\gamma & \gamma -1\\ 3 & 1\end{pmatrix}\in\Ncal_{e_2} \quad\text{and}\quad 
		\nabla v \restrict{T_8} = \nabla v \restrict{T_4} = \frac{1}{2}\begin{pmatrix}1+\gamma &-3+\gamma \\ 1 & 1\end{pmatrix}\in\Ncal_{e_1}.
	\end{align}

	We now define the polycrystal $(\Omega, R_\ast)$ as follows: In light of \eqref{S}-\eqref{T2468}, let the grains of the polycrystal be
	\begin{align*}
		\Omega_1 ={\rm int}(T_2\cup T_3),\quad \Omega_2={\rm int}(T_6\cup T_7), \quad \Omega_3={\rm int}(T_1\cup T_5\cup T_4\cup T_8\cup S),
	\end{align*}
	and let the orientations be given by
	\begin{align}\label{texture_example}
		R_\ast &= \Id\ONE_{\Omega_3} + R_{\frac{\pi}{2}}\ONE_{\Omega_1\cup\Omega_2}.
	\end{align}
	\medskip	

	Overall, the procedure above produces a finitely piecewise affine solution $v$ to the relaxed problem \eqref{relaxed} for the polycrystal $(\Omega,R_\ast)$ with texture $R_\ast$ as in \eqref{texture_example} and the boundary value $v = F_\gamma x$ on $\partial\Omega$ with
	\begin{align*}
		F_\gamma =\frac{1}{5}\begin{pmatrix} 3\gamma + 4 & 4\gamma - 3 \\ 3 & 4 \end{pmatrix} \quad \text{for }\gamma\in [-\sqrt{3}+1, \sqrt{3}-1], 
	\end{align*}
	see Figure \ref{fig:sheared_square1} for illustration. 
	Note that $F_\gamma\in\SO(2)$ if and only if $\gamma=0$, so that, in combination with Proposition~\ref{prop:affine_lipschitz}, any $\gamma\in \R$ with $0<|\gamma|<\sqrt{3}-1$ gives rise to a solution of~\eqref{unrelaxed} with $F\notin \SO(2)$. 
	This shows that $\Tcal_{\Ncal}(R_\ast(\Omega))=\SO(2)\subsetneq \Fcal_{\Mcal}(\Omega, R_\ast)$ and concludes the proof.
	\end{proof}
	
	\begin{figure}[h!]
		\centering
		\begin{tikzpicture}
			\begin{scope}[shift={(-7,1)}]
				\draw [fill=Dandelion, smooth, line join=round, very thin](1,1)--++(1,1)--++(1,-1)--++(-1,-1)--cycle; //S
		
				\draw [fill=Cerulean, very thin](1,1) --++ (1,1)--++(-1,1)--cycle; //T1
				\draw [fill=Cerulean, very thin](2,0)--++(1,-1)--++(0,2)--cycle; //T5
				\draw [thick, ->] (2.6,-0.2) --++(90:0.4);
				\draw [thick, ->](1.4,1.8)--++(90:0.4);
				\draw [thick, ->](1.8,1) --++(0:0.4); 
				
				\draw [fill=Cerulean, smooth, line join=round, very thin](0,0) --(1,1)--++(0,2)--(0,0); //T2
				\draw [fill=Cerulean, smooth, line join=round, very thin](3,-1)--++(1,3)--++(-1,-1)--cycle; //T6
				
				\draw [fill=Dandelion, smooth, line join=round, very thin](0,0)--(2,0)--(1,1)--(0,0); //T3
				\draw [fill=Dandelion, smooth, line join=round, very thin](2,2)--++(2,0)--++(-1,-1)--cycle; //T7
				
				\draw [fill=Dandelion, smooth, line join=round, very thin](2,2)--++(2,0)--++(-3,1)--cycle; //T4
				\draw [fill=Dandelion, smooth, line join=round, very thin](2,0)--++(1,-1)--(0,0)--(2,0); //T8
				\draw (1,3) node [anchor = south west] {$\Omega$};
				\draw [thick] (0,0)--++(1,3)--++(3,-1)--++(-1,-3)--cycle;
				\draw [thick] (0,0)--++(2,2)--++(-1,1);
				\draw [thick] (3,-1)--++(-1,1)--++(2,2);
				
				\draw [thick, ->](4.5,2.5) to [out=45,in=135] (6.5,2.5);
				\draw (5.5,3.25) node {$v$};
			\end{scope}
			\begin{scope}
				\draw [fill=Dandelion, smooth, line join=round, very thin](1,2)--++(1.5,1)--++(0.5,-1)--++(-1.5,-1)--cycle;//S
					
				\draw [fill=Dandelion, smooth, line join=round, very thin](1,2)--++(0.5,-1)--++(-1.5,-1)--cycle; //T1
				\draw [fill=Dandelion, smooth, line join=round, very thin](2.5,3)--++(1.5,1)--++(-1,-2)--cycle; //T5
				
				\draw [fill=Cerulean, smooth, line join=round, very thin](0,0)--++(0.5,3)--+(0.5,-1)--cycle; //T2
				\draw [fill=Cerulean, smooth, line join=round, very thin](3,2)--++(1,2)--++(-0.5,-3)--cycle; //T6
				
				\draw [fill=Cerulean, smooth, line join=round, very thin](1,2)--++(-0.5,1)--++(2,0)--cycle; //T3
				\draw [fill=Cerulean, smooth, line join=round, very thin](3,2)--++(0.5,-1)--++(-2,0)--cycle; //T7
				
				\draw [fill=Dandelion, smooth, line join=round, very thin](0.5,3)--++(3.5,1)--++(-1.5,-1)--cycle; //T4
				\draw [fill=Dandelion, smooth, line join=round, very thin](0,0)--++(1.5,1)--++(2,0)--cycle; //T8
				
				\draw [thick](0,0)--++(0.5,3)--++(3.5,1)--++(-0.5,-3)--cycle;
				\draw [thick](0,0)--++(1,2)--++(1.5,1)--++(-2,0);
				\draw [thick](3.5,1)--++(-2,0)--++(1.5,1)--++(1,2);
				\draw (1,3.5) node {$v(\Omega)$};
			\end{scope}
		\end{tikzpicture}
		\caption{The finitely piecewise affine solution $v$ to \eqref{relaxed} with $F=F_{\frac{1}{2}}$. The orange area describes the grain with slip direction $e_1$, while the blue ones represent the grains with slip direction $e_2$.}\label{fig:sheared_square1}
	\end{figure}
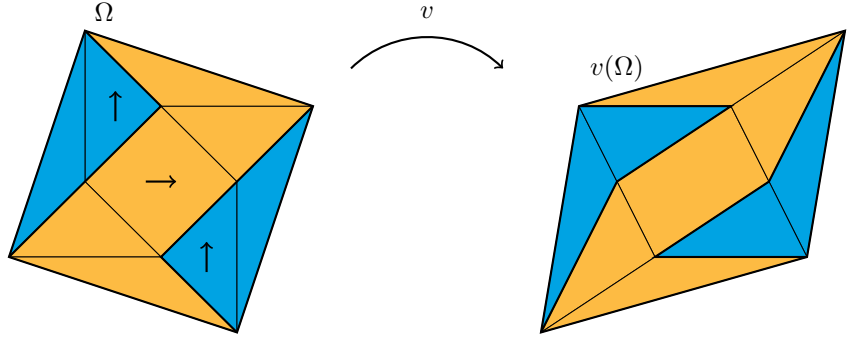

\begin{remark}[Comparison with the rotated-square construction]
	While our geometric framework  is similar to that of \cite{CKZ17, Pom10}, the overall design strategy is different, given that we need to accommodate here the homogeneous differential inclusion \eqref{hom_relaxed} with a non-trivial affine boundary condition. 
	
	Unlike the rotated-square construction of~\cite{CKZ17, Pom10}, our sheared-square construction converts any additional rotation of the center square into a rotation of the surrounding triangles $T_1,\ldots,T_8$ by the same angle, considering that the vertices of $\Omega$ are not fixed.
	The boundary of $\Omega$ is thus rotated in the same way.
	More precisely, replacing $\nabla v\restrict{S}$ by $R\nabla v\restrict{S}$ for a rotation $R\in\SO(2)$ yields the boundary value $RF_\gamma$ instead of $F_\gamma$.\medskip
\end{remark}

\section*{Appendix}
\renewcommand{\thesection}{A}
\stepcounter{section}
\renewcommand{\theequation}{A.\arabic{equation}}
\setcounter{equation}{0}

\begin{lemma}\label{lem:random_crystals_app}
	Let $k\in \N$ and $\Tcal_k=\big \{(\theta_1,\ldots,\theta_{k}) \in (0,\pi)^k : \textstyle\bigcap_{i=1}^{k} \Ncal R_{\theta_i}^T \cap \Ncal = \SO(2)\big\}$.  Then it holds for 
	the $k$-dimensional Lebesgue measure of $\Tcal_k$ that
	\begin{align}\label{probability}
		\lambda_k(\Tcal_k) = \pi^k\Bigl(1 - \frac{k+1}{2^k}\Bigr). 
	\end{align}
\end{lemma}
\begin{proof} 
	For integers $1\leq j\leq i$, let $\Sigma_{i,j}$ be the set of all injective functions $\sigma : \{1,\ldots,j\}\to\{1,\ldots,i\}$ and observe that $\Sigma_{i,j}$ consists of exactly $\frac{i!}{(i-j)!}$ elements; also, set $\Sigma_{i,0}=\emptyset$.
	
	We establish \eqref{probability} by computing the measure of the complement $\Tcal_k^c$ of $\Tcal_k$ in $(0, \pi)^k$.
	In view of Corollary \ref{cor:trivial_taylorbound}, $\Tcal_k^c$ can be expressed as the disjoint union 
	\begin{align}\label{partition_probability}
		\Tcal_k^c = \bigcup_{l=0}^{k} T_{k,l} \cup N,
	\end{align}
	where $N\subset \R^k$ is a set of zero $k$-dimensional Lebesgue measure, and
	\begin{align}\label{T_k,l}
		\begin{split}T_{k,l} := \{&(\theta_1,\ldots,\theta_{k})\in(0,\pi)^k : \theta_i<\tfrac{\pi}{2} \text{ and } \theta_j>\tfrac{\pi}{2} \text{ with } \theta_j - \theta_i >\tfrac{\pi}{2}\\
		&\text{ for all }  i\in\sigma(\{1,\ldots,l\}), j\in\{1,\ldots,k\}\setminus\sigma(\{1,\ldots,l\})\text{ and all $\sigma\in\Sigma_{k,l}$}\}\end{split}
	\end{align}
	for $l\in\{0,\ldots, k\}$. 
	Note that $T_{k,0}$ and $T_{k,k}$ correspond to the cases $\theta_{i}>\frac{\pi}{2}$ and $\theta_{i}<\frac{\pi}{2}$ for all $i\in \{1,\ldots,k\}$, respectively.

	To calculate $\lambda_k(\Tcal_{k}^c)$, it suffices to determine the measures of $T_{k,l}$ as in \eqref{T_k,l} and exploit \eqref{partition_probability}. 
	It is easy to see that
	\begin{align}\label{boundarycases}
		\lambda_k(T_{k,0}) = \lambda_k(T_{k,k}) = \Bigl(\frac{\pi}{2}\Bigr)^k.
	\end{align}
		For $l\in\{1,\ldots,k-1\}$, one obtains that
	\begin{align}\label{innercases}
		\lambda_k(T_{k,l}) &= {{k}\choose{l}} \int_{(0,\frac{\pi}{2})^{l}}\int_{\bigl(\frac{\pi}{2} + \max\{\theta_1,\ldots,\theta_{l}\},\pi\bigr)^{k-l}} \dd(\theta_{l+1},\ldots \theta_{k}) \dd(\theta_1,\ldots,\theta_{l})\nonumber\\
		&=	{{k}\choose{l}} \int_{(0,\frac{\pi}{2})^{l}}\bigl(\tfrac{\pi}{2} - \max\{\theta_1,\ldots,\theta_{l}\}\bigr)^{k-l}\dd(\theta_1,\ldots,\theta_{l})\nonumber\\
		&= {{k}\choose{l}} l! \int_0^{\tfrac{\pi}{2}}\int_0^{\theta_{l}}\ldots \int_0^{\theta_{2}} \big(\tfrac{\pi}{2} - \theta_l\big)^{k-l} \dd \theta_1 \ldots \dd \theta_{l-1} \dd \theta_l\nonumber\\
		&= {{k}\choose{l}} l \int_0^{\tfrac{\pi}{2}}\big(\tfrac{\pi}{2} - \theta_l\big)^{k-l}\theta_l^{l-1}\dd \theta_l= {{k}\choose{l}} l \frac{(l-1)!(k-l)!}{k!}\Bigl(\frac{\pi}{2}\Bigr)^k =\Bigl( \frac{\pi}{2}\Bigl)^k,
	\end{align}
	where the third equality uses that $(0,\frac{\pi}{2})^l$ can (up to a null set) be split disjointly into $l!$ sets with equal measure 
	\begin{align*}
		\bigcup_{\sigma\in\Sigma_{l,l}} \{(\theta_1,\ldots,\theta_l) \in (0, \tfrac{\pi}{2})^l: \theta_{\sigma(1)}< \ldots < \theta_{\sigma(l)}\};
	\end{align*}
	moreover, the last integral in~\eqref{innercases} is solved via integration by parts applied $(k-l)$ times. 
	
	The desired identity \eqref{probability} then follows from \eqref{partition_probability}, in view of \eqref{boundarycases} and \eqref{innercases}.
\end{proof}

\begin{lemma}\label{lem:symmetric}
	It holds for any $R\in\SO(2)\setminus\{\pm \Id\}$ that
	\begin{align*}
		\bigcap_{k=0}^\infty \Ncal (R^k)^T=\SO(2).
	\end{align*}
\end{lemma} 

\begin{proof}
	In light of Remark~\ref{rem:first_angle} a), one may assume that
	$R=R_\ffi$ with $\ffi\in (0,\pi)$.  We set
	\begin{align*}
		 \theta_j=j\ffi  - \lfloor \tfrac{j\ffi}{\pi} \rfloor  \pi \quad \in [0, \pi)
	\end{align*}
	for any $j\in\N$, where $\lfloor t \rfloor$ denotes the largest integer below $t\in \R$.
	If one can find $k, l\in \N$ such that
	\begin{align}\label{bar_tilde_theta}
		0\leq \theta_k<\theta_l<\pi \quad\text{and}\quad \frac{\pi}{2}\in[\theta_k, \theta_l]\quad \text{and}\quad \theta_l -  \theta_k \leq \frac{\pi}{2},
	\end{align}
	then the statement follows immediately from
	\begin{align*}
		\SO(2)\subset  \bigcap_{k=0}^\infty \Ncal (R^k)^T \subset \Ncal\cap \Ncal R^T_{\theta_k}\cap \Ncal R^T_{\theta_l} = \SO(2),
	\end{align*}
	where the last identity is a consequence of Corollary~\ref{cor:trivial_taylorbound}. 

	To see~\eqref{bar_tilde_theta}, let us write $(0, \pi)$ as the disjoint union 
	\begin{align*}
		(0,\pi) = \bigcup_{m=2}^\infty I_m \cup (0,\tfrac{\pi}{2})\quad \text{with} \quad I_m:= \big[(1-\tfrac{1}{2^{m-1}})\pi,(1-\tfrac{1}{2^{m}})\pi\big). 
	\end{align*}
	If $\ffi\in (0,\frac{\pi}{2})$, we take $k=\lfloor\frac{\pi}{2\varphi}\rfloor$ and $l=\lfloor\frac{\pi}{2\varphi}\rfloor+1$, observing that $0<\theta_k=\lfloor \frac{\pi}{2\varphi}\rfloor\varphi< \frac{\pi}{2} <(\lfloor \frac{\pi}{2\varphi}\rfloor+1)\varphi =\theta_l<\pi$ and $\theta_l-\theta_k=\varphi<\frac{\pi}{2}$. 
	For $\ffi\in I_m$ with $m\geq 2$, let $l=2^{m-2}$ and $k=2^{m-1}$. 
	Then, $\theta_k=2^{m-1}\varphi - (2^{m-1}-1)\pi \in [0,\frac{\pi}{2})$ and $\theta_l = 2^{m-2}\ffi - (2^{m-2}-1)\pi\in[\frac{\pi}{2}, \frac{3}{4}\pi)$ with
	\begin{align*}
		\theta_l - \theta_k = 2^{m-2}(\pi-\ffi) \leq \frac{2^{m-2}}{2^{m-1}}\pi = \frac{\pi}{2}.
	\end{align*}
\end{proof}

\section*{Acknowledgments}
The authors would like to thank John Ball for sharing helpful insight into the generalized Hadamard jump condition. 
CK acknowledges the support from the Dutch Research Council (NWO) through the project TOP2.17.01. Most of this work was done while CK was affiliated with Utrecht University and supported by the Westerdijk Fellowship program.


\bibliographystyle{abbrv}
\bibliography{polycrystals}
\end{document}